\documentclass[
	english
]{scrartcl}

\usepackage[T1]{fontenc}
\usepackage[utf8]{inputenc}

\usepackage{amsmath,amsfonts,amsthm,amssymb}
\allowdisplaybreaks

\usepackage[colorlinks,citecolor=blue]{hyperref}
\usepackage{doi,natbib}
\usepackage{cancel}
\usepackage{verbatim}
\usepackage{float}

\usepackage[shortlabels]{enumitem}

\usepackage{relsize}
\usepackage[utf8]{inputenc}

\usepackage[colorlinks,citecolor=blue]{hyperref}
\usepackage[figure]{hypcap}
\usepackage{doi,natbib}

\usepackage{changes}

\usepackage{preamble}

\usepackage{marginnote}

\usepackage[capitalise,nameinlink]{cleveref}

\usepackage{graphicx}
\graphicspath{{./img/}}
\usepackage{subcaption} 

\newcommand\norm[2]{\left\Vert#1\right\Vert_{#2}}

\newcommand\N{\mathbb{N}}
\newcommand\R{\mathbb{R}}

\newcommand{\supp}{\operatorname{supp}}

\DeclareMathAlphabet{\mathpzc}{OT1}{pzc}{m}{it}

\newtheorem{theorem}{Theorem}[section]
\newtheorem{lemma}[theorem]{Lemma}

\newtheorem{assumption}[theorem]{Assumption}
\newtheorem{corollary}[theorem]{Corollary}

\newtheorem{definition}[theorem]{Definition}
\newtheorem{example}[theorem]{Example}
\newtheorem{conjecture}[theorem]{Conjecture}

\definecolor{mygreen}{rgb}{0.0,0.7,0.0}
\definecolor{myred}{rgb}{0.7,0.0,0.3}
\definecolor{mybrown}{rgb}{0.5,0.5,0.0}

\begin{document}

\title{Relaxation schemes for mathematical programs with switching constraints}
\author{%
	Christian Kanzow%
	\footnote{%
		University of W\"urzburg,
		Institute of Mathematics,
		97074 W\"urzburg,
		Germany,
		\email{kanzow@mathematik.uni-wuerzburg.de},
		\url{https://www.mathematik.uni-wuerzburg.de/\~kanzow/},
		ORCID: 0000-0003-2897-2509
		}
	\and
	Patrick Mehlitz%
	\footnote{%
		Brandenburgische Technische Universität Cottbus--Senftenberg,
		Institute of Mathematics,
		03046 Cottbus,
		Germany,
		\email{mehlitz@b-tu.de},
		\url{https://www.b-tu.de/fg-optimale-steuerung/team/dr-patrick-mehlitz},
		ORCID: 0000-0002-9355-850X%
		}
	\and
	Daniel Steck%
	\footnote{%
    	University of W\"urzburg,
		Institute of Mathematics,
		97074 W\"urzburg,
		Germany,
		\email{daniel.steck@mathematik.uni-wuerzburg.de},
		\url{https://www.mathematik.uni-wuerzburg.de/\~steck/},
		ORCID: 0000-0001-5335-5428
		}
	}

\publishers{}
\maketitle

\begin{abstract}
	 Switching-constrained optimization problems form a difficult class of mathematical programs 
	 since their feasible set is almost disconnected while standard constraint qualifications
	 are likely to fail at several feasible points. That is why the application of standard 
	 methods from nonlinear programming does not seem to be promising in order to solve such
	 problems. In this paper, we adapt the relaxation method from Kanzow and Schwartz (SIAM J.~Optim., 23(2):770--798, 2013)
	 for the numerical treatment of mathematical programs with complementarity constraints to
	 the setting of switching-constrained optimization. It is shown that the proposed method
	 computes M-stationary points under mild assumptions. Furthermore, we comment on 
	 other possible relaxation approaches which can be used to tackle mathematical programs with
	 switching constraints. As it turns out, adapted versions of Scholtes' global relaxation scheme
	 as well as the relaxation scheme of Steffensen and Ulbrich only find W-stationary points of
	 switching-constrained optimization problems in general. Some computational experiments visualize
	 the performance of the proposed relaxation method.
\end{abstract}


\begin{msc}	
	65K05, 90C30, 90C33.
\end{msc}

\section{Introduction}\label{sec:introduction}

This paper is dedicated to so-called \emph{mathematical programs with switching constraints}, MPSCs for short. These are optimization problems of the form
\begin{equation}\label{eq:MPSC}\tag{MPSC}
	\begin{aligned}
		f(x)&\,\rightarrow\,\min&&&\\
		g_i(x)&\,\leq\,0,&\qquad&i\in\mathcal M,&\\
		h_j(x)&\,=\,0,&\qquad&j\in\mathcal P,&\\
		G_l(x)H_l(x)&\,=\,0,&\qquad&l\in\mathcal Q,&
	\end{aligned}
\end{equation}
where $\mathcal{M}:=\{1,\ldots,m\}$, $\mathcal{P}:=\{1,\ldots,p\}$, $\mathcal{Q}:=\{1,\ldots,q\}$ are index sets and the functions 
$f,g_i,h_j,G_l,H_l\colon\R^n\to\R$  are continuously differentiable for all $i\in\mathcal M$, $j\in\mathcal P$, and $l\in\mathcal Q$. 
For brevity, $g\colon\R^n\to\R^m$, $h\colon\R^n\to\R^p$, $G\colon\R^n\to\R^q$, and $H\colon\R^n\to\R^q$
are the mappings which possess the component functions 
$g_i$ ($i\in\mathcal M$), $h_j$ ($j\in\mathcal P$), $G_l$ ($l\in\mathcal Q$), 
and $H_l$ ($l\in\mathcal Q$), respectively.
The last $q$ constraints in \eqref{eq:MPSC} force $G_l(x)$ or $H_l(x)$ to be zero for all $l\in\mathcal Q$, which gives rise to the terminology ``switching constraints''.

Switching structures appear frequently in the context of optimal control, see 
\cite{ClasonRundKunisch2017,Gugat2008,HanteSager2013,Liberzon2003,Seidman2013,WangYan2015,Zuazua2011},
and the references therein, or as a reformulation of so-called \emph{either-or constraints}, 
see \cite[Section~7]{Mehlitz2018}.
Naturally, \eqref{eq:MPSC} is related to other problem classes from disjunctive programming 
such as \emph{mathematical programs with complementarity constraints}, MPCCs for short, 
see \cite{LuoPangRalph1996,OutrataKocvaraZowe1998}, 
or \emph{mathematical programs with vanishing constraints}, 
MPVCs for short, see \cite{AchtzigerKanzow2008,HoheiselKanzow2008}. 
Indeed, similarly to MPCCs and MPVCs, standard constraint qualifications are likely to be violated 
at the feasible points of \eqref{eq:MPSC}. Recently, stationarity conditions and 
constraint qualifications for \eqref{eq:MPSC} were introduced in \cite{Mehlitz2018}.

Here, we focus on the computational treatment of \eqref{eq:MPSC}. 
Clearly, standard methods from nonlinear programming may run into difficulties
when applied to \eqref{eq:MPSC} due to two reasons:
first, the feasible set of \eqref{eq:MPSC} is likely to be disconnected or at least \emph{almost}
disconnected. Secondly, standard regularity conditions like the Mangasarian--Fromovitz constraint qualification,
MFCQ for short, are likely to fail at the feasible points of \eqref{eq:MPSC} under mild assumptions,
see \cite[Lemma~4.1]{Mehlitz2018}. Similar issues appear in the context of MPCCs and MPVCs where several
different relaxation schemes were introduced to overcome these shortcomings, 
see \cite{HoheiselKanzowSchwartz2012,HoheiselKanzowSchwartz2013} and the references therein.
Basically, the idea is to relax the \emph{irregular} constraints using a relaxation parameter such that
the resulting surrogate problems are (regular) standard nonlinear problems which can be tackled by common methods.
The relaxation parameter is then iteratively reduced to zero and, in each iteration, a Karush--Kuhn--Tucker (KKT) 
point of the surrogate problem is computed.
Ideally, the resulting sequence possesses a limit point and, under some problem-tailored constraint qualification, 
this point satisfies a suitable stationarity condition.
Furthermore, it is desirable that the relaxed problems satisfy standard constraint qualifications in a neighborhood 
of the limit point under reasonable assumptions.

In this paper, we show that the relaxation scheme from \cite{KanzowSchwartz2013}, which was designed
for the numerical investigation of MPCCs, can be adapted for the computational treatment of \eqref{eq:MPSC}. 
Particularly, it will be shown that the modified method can be used to find M-stationary points of \eqref{eq:MPSC}.
By means of examples it is demonstrated that other relaxation methods which are well known from the theory of MPCCs
only yield W-stationary points of \eqref{eq:MPSC}. We present the results of some numerical experiments which
show the performance of the proposed method.

The remaining part of the paper is structured as follows. In Section~\ref{sec:notation_preliminaries}, we describe the general notation used throughout the paper and recall some fundamental theory on nonlinear programming and switching constraints. Section~\ref{sec:relaxation_scheme_KS} is dedicated to the main relaxation approach and contains various properties of the resulting algorithm as well as convergence results. In Section~\ref{sec:other_relaxations}, we describe how some common regularization methods from MPCCs can be carried over to the switching-constrained setting, and discuss the convergence properties of the resulting algorithms. Section~\ref{sec:numerical_results} contains some numerical applications, including either-or constraints, switching-constrained optimal control, and semi-continuous optimization problems arising in portfolio optimization. We conclude the paper with some final remarks in Section~\ref{sec:final_results}.

\section{Notation and preliminaries}\label{sec:notation_preliminaries}

\subsection{Basic notation}

The subsequently introduced tools of variational analysis can be 
found in \cite{RockafellarWets1998}.

For a nonempty set $A\subset\R^n$, we call
\begin{equation*}
    A^\circ:=\{y\in\R^n\,|\,\forall x\in A\colon\,x\cdot y\leq 0\}
\end{equation*}
the \emph{polar cone} of $A$. Here, $x\cdot y$ denotes the Euclidean inner product
of the two vectors $x,y\in\R^n$. It is well known that $A^\circ$
is a nonempty, closed, convex cone. For any two sets $B_1,B_2\subset\R^n$, the
polarization rule $(B_1\cup B_2)^\circ=B_1^\circ\cap B_2^\circ$ holds by definition.
The polar of a polyhedral cone can be characterized by means of, e.g.,
Motzkin's theorem of alternatives. Note that we interpret the relations
$\leq$ and $\geq$ for vectors componentwise.
\begin{lemma}\label{lem:polar_of_polyhedral_cone}
    For matrices $\mathbf C\in\R^{m\times n}$ and $\mathbf D\in\R^{p\times n}$, let $\mathcal K\subset\R^n$ be the polyhedral cone
\begin{equation*}
    \mathcal K:=\{d\in\R^n\,|\,\mathbf Cd\leq 0,\,\mathbf Dd=0\}.
\end{equation*}
    Then $\mathcal K^\circ=\{\mathbf C^\top\lambda+\mathbf D^\top\rho\,|\, \lambda\in\R^m,\,\lambda\geq 0,\,\rho\in\R^p\}$.
\end{lemma}

Let $A\subset\R^n$ be a nonempty set and $\bar x\in A$. Then the closed cone
\begin{equation*}
	\mathcal T_A(\bar x):=
	\left\{
		d\in\R^n\,\middle|\,
		\exists\{x_k\}_{k\in\N}\subset A\,\exists\{\tau_k\}_{k\in\N}\subset\R_+\colon\,x_k\to\bar x,\,\tau_k\to 0,\, (x_k-\bar x)/\tau_k\to d
	\right\}
\end{equation*}
is called \emph{tangent} or \emph{Bouligand cone} to $A$ at $\bar x$. Here, $\R_+:=\{r\in\R\,|\,r>0\}$ denotes the set of all
positive reals.

The union $\{v^1,\ldots,v^r\}\cup\{w^1,\ldots,w^s\}$ of sets $\{v^1,\ldots,v^r\},\{w^1,\ldots,w^s\}\subset\R^n$
is called \emph{positive-linearly dependent} if there exist vectors $\alpha\in\R^r$, $\alpha\geq 0$, and $\beta\in\R^s$
which do not vanish at the same time such that
\begin{equation*}
    0=\sum\limits_{i=1}^r\alpha_iv^i+\sum\limits_{j=1}^s\beta_jw^j.
\end{equation*}
Otherwise, $\{v^1,\ldots,v^r\}\cup\{w^1,\ldots,w^s\}$ is called \emph{positive-linearly independent}. Clearly, 
if the set $\{v^1,\ldots,v^r\}$ is empty, then the above definitions reduce to linear dependence and independence,
respectively. The following lemma will be useful in this paper; its proof is similar to that of
\cite[Proposition~2.2]{QiWei2000} and therefore omitted.

\begin{lemma}\label{lem:local_stability_of_positive_linear_independence}
Let $\{v^1,\ldots,v^r\},\{w^1,\ldots,w^s\}\subset\R^n$ be given sets whose union
 $\{v^1,\ldots,v^r\}\cup\{w^1,\ldots,w^s\}$ 
is positive-linearly independent. 
Then there exists $\varepsilon>0$ such that, for all vectors 
$\tilde v^1,\ldots,\tilde v^r,\tilde w^1,\ldots,\tilde w^s\in\{z\in\R^n\,|\,\norm{z}{2}\leq\varepsilon\}$,
the union $\{v^1+\tilde v^1,\ldots,v^r+\tilde v^r\}\cup\{w^1+\tilde w^1,\ldots,w^s+\tilde w^s\}$ is positive-linearly
independent.
\end{lemma}

For some vector $z\in\R^n$ and an index set $I\subset\{1,\ldots,n\}$, $z_I\in\R^{|I|}$ denotes the
vector which results from $z$ by deleting all $z_i$ with $i\in\{1,\ldots,n\}\setminus I$.
Finally, let us mention that $\supp z:=\{i\in\{1,\ldots,n\}\,|\,z_i\neq 0\}$ is called the \emph{support} of 
the vector $z\in\R^n$.

\subsection{Standard nonlinear programs}

Here, we recall some fundamental constraint qualifications from standard nonlinear programming, see, e.g., \cite{BazaraaSheraliShetty1993}. Therefore, we consider the nonlinear program
\begin{equation}\label{eq:nonlinear_program}\tag{NLP}
	\begin{aligned}
		f(x)&\,\rightarrow\,\min&&&\\
		g_i(x)&\,\leq\,0,&\qquad&i\in\mathcal M,&\\
		h_j(x)&\,=\,0,&&j\in\mathcal P,&
	\end{aligned}
\end{equation}
i.e., we forget about the switching constraints in \eqref{eq:MPSC} for a moment.

Let $\tilde X\subset\R^n$ denote the feasible set of \eqref{eq:nonlinear_program} and fix some
point $\bar x\in \tilde X$. Then
\[
	I^g(\bar x):=\{i\in\mathcal M\,|\,g_i(\bar x)=0\}
\]
is called the index set of active inequality constraints at $\bar x$. 
Furthermore, the set
\[
	\mathcal L_{\tilde X}(\bar x)
	:=
	\left\{
		d\in\R^n\,\middle|\,
			\begin{aligned}
				\nabla g_i(\bar x)\cdot d&\,\leq\,0&&i\in I^g(\bar x)\\
				\nabla h_j(\bar x)\cdot d&\,=\,0&&j\in\mathcal P
			\end{aligned}
	\right\}
\]
is called the linearization cone to $\tilde X$ at $\bar x$.
Obviously, $\mathcal L_{\tilde X}(\bar x)$ is a polyhedral cone, and thus closed and convex.
It is well known that
$\mathcal T_{\tilde X}(\bar x)\subset \mathcal L_{\tilde X}(\bar x)$ is always
satisfied. The converse inclusion generally only holds under some constraint qualification.

In the definition below, we recall several standard constraint qualifications which are applicable to \eqref{eq:nonlinear_program}.
\begin{definition}\label{def:CQs}
	Let $\bar x\in \R^n$ be a feasible point of \eqref{eq:nonlinear_program}.
	Then $\bar x$ is said to satisfy the
\begin{enumerate}[(a)]
    \item \emph{linear independence constraint qualification (LICQ)} if the following vectors are linearly independent:
\begin{equation}\label{eq:NLP_vectors_for_CQs}
    \{\nabla g_i(\bar x)\,|\,i\in I^g(\bar x)\}\cup\{\nabla h_j(\bar x)\,|\,j\in\mathcal P\}.
\end{equation}
    \item \emph{Mangasarian--Fromovitz constraint qualification (MFCQ)} if the vectors in \eqref{eq:NLP_vectors_for_CQs} are positive-linearly independent.
    \item \emph{constant positive linear dependence condition (CPLD)} if, for any sets $I\subset I^g(\bar x)$ and $J\subset\mathcal P$ such that the gradients
\begin{equation*}
    \{\nabla g_i(\bar x)\,|\,i\in I\}\cup\{\nabla h_j(\bar x)\,|\,j\in J\}
\end{equation*}
    are positive-linearly dependent, there exists a neighborhood $U\subset\R^n$ of $\bar x$ such that the gradients
\begin{equation*}
    \{\nabla g_i(x)\,|\,i\in I\}\cup\{\nabla h_j(x)\,|\,j\in J\}
\end{equation*}
    are linearly dependent for all $x\in U$.
    \item \emph{Abadie constraint qualification (ACQ)} if $\mathcal T_{\tilde X}(\bar x)=\mathcal L_{\tilde X}(\bar x)$.
    \item \emph{Guignard constraint qualification (GCQ)} if	$\mathcal T_{\tilde X}(\bar x)^\circ=\mathcal L_{\tilde X}(\bar x)^\circ$.
\end{enumerate}
\end{definition}

Note that the following relations hold between the constraint qualifications from Definition \ref{def:CQs}:
\[
	\text{LICQ}\,\Longrightarrow\,\text{MFCQ}\,\Longrightarrow\,\text{CPLD}\,\Longrightarrow\,
	\text{ACQ}\,\Longrightarrow\,\text{GCQ},
\]
see \cite[Section~2.1]{HoheiselKanzowSchwartz2013} for some additional information.

It is well known that the validity of GCQ at some local minimizer $\bar x\in\R^n$ of \eqref{eq:nonlinear_program}
implies that the KKT conditions
\[
	\begin{split}
		&0=\nabla f(\bar x)+\sum\limits_{i\in I^g(\bar x)}\lambda_i\nabla g_i(\bar x)+\sum\limits_{j\in\mathcal P}\rho_j\nabla h_j(\bar x),\\
		&\forall i\in I^g(\bar x)\colon\,\lambda_i\geq 0
	\end{split}
\]
provide a necessary optimality condition. Thus, the same holds for the stronger constraint qualifications ACQ, CPLD, MFCQ, and LICQ.

\subsection{Mathematical programs with switching constraints}

The statements of this section are taken from \cite{Mehlitz2018}.
Let $X\subset\R^n$ denote the feasible set of \eqref{eq:MPSC} and fix a point $\bar x\in X$. Then the index sets
\[
	\begin{split}
		I^G(\bar x)&:=\{l\in\mathcal Q\,|\,G_l(\bar x)=0,\,H_l(\bar x)\neq 0\},\\
		I^H(\bar x)&:=\{l\in\mathcal Q\,|\,G_l(\bar x)\neq 0,\,H_l(\bar x)=0\},\\
		I^{GH}(\bar x)&:=\{l\in\mathcal Q\,|\,G_l(\bar x)=0,\,H_l(\bar x)= 0\}
	\end{split}
\]
form a disjoint partition of $\mathcal Q$. It is easily seen that MFCQ (and thus LICQ) cannot hold for \eqref{eq:MPSC} at $\bar x$ if $I^{GH}(\bar x)\neq\varnothing$. Taking a look at the associated linearization cone
\[
	\mathcal L_X(\bar x)
		=
		\left\{
			d\in\R^n\,
		\middle|\,
			\begin{aligned}
				\nabla g_i(\bar x)\cdot d&\,\leq\,0&&i\in I^g(\bar x)\\
				\nabla h_j(\bar x)\cdot d&\,=\,0&&j\in \mathcal P\\
				\nabla G_k(\bar x)\cdot d&\,=\,0&&k\in I^G(\bar x)\\
				\nabla H_k(\bar x)\cdot d&\,=\,0&&k\in I^H(\bar x)
			\end{aligned}
		\right\},
\]
which is always convex, one can imagine that ACQ is likely to fail as well if $I^{GH}(\bar x)\neq\varnothing$ since, in the latter situation, $\mathcal T_X(\bar x)$ might be nonconvex. Note that GCQ may hold for \eqref{eq:MPSC} even
in the aforementioned context.

Due to the inherent lack of regularity, stationarity conditions for \eqref{eq:MPSC} which are weaker than the associated 
KKT conditions were introduced. 
\begin{definition}\label{def:stationarities}
    A feasible point $\bar x\in X$ of \eqref{eq:MPSC} is called
\begin{enumerate}[(a)]
    \item \emph{weakly stationary (W-stationary)} if there exist multipliers $\lambda_i$ ($i\in I^{g}(\bar x)$), $\rho_j$ ($j\in\mathcal P$), 
    $\mu_l$ ($l\in\mathcal Q$), and $\nu_l$ ($l\in\mathcal Q$) which solve the following system:
			\[
				\begin{split}
					&0=\nabla f(\bar x)+\sum\limits_{i\in I^g(\bar x)}\lambda_i\nabla g_i(\bar x)
										+\sum\limits_{j\in \mathcal P}\rho_j\nabla h_j(\bar x)\\
					&\qquad\qquad
						+\sum\limits_{l\in \mathcal Q}\bigl[\mu_l\nabla G_l(\bar x)+\nu_l\nabla H_l(\bar x)\bigr],\\
					&\forall i\in I^g(\bar x)\colon\,\lambda_i\geq 0,\\
					&\forall l\in I^H(\bar x)\colon\,\mu_l=0,\\
					&\forall l\in I^G(\bar x)\colon\,\nu_l=0.
				\end{split}
			\]
    \item \emph{Mordukhovich-stationary (M-stationary)} if it is W-stationary and the associated multipliers additionally satisfy
			\[
				\forall l\in I^{GH}(\bar x)\colon\,\mu_l\nu_l=0.
			\]
    \item \emph{strongly stationary (S-stationary)} if it is W-stationary while the associated multipliers additionally satisfy
			\[
				\forall l\in I^{GH}(\bar x)\colon\,\mu_l=0\,\land\,\nu_l=0.
			\]
	\end{enumerate}
\end{definition}

Clearly, the following implications hold:
\[
	\text{S-stationarity}\;\Longrightarrow\;\text{M-stationarity}\;\Longrightarrow\;\text{W-stationarity.}
\]
Moreover, the KKT conditions of \eqref{eq:MPSC} are equivalent to the S-stationarity conditions from Definition \ref{def:stationarities}.
One may check Figure \ref{fig:stationarities} for a geometric interpretation of the Lagrange multipliers
associated with the switching conditions from $I^{GH}(\bar x)$.
\begin{figure}[h]\centering
\includegraphics[width=4.0cm]{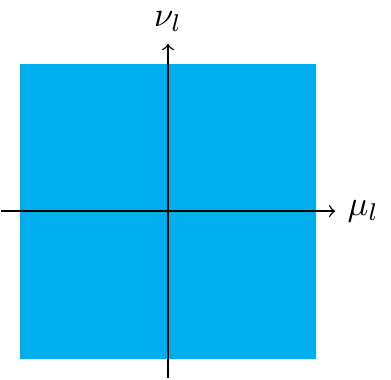}
\includegraphics[width=4.0cm]{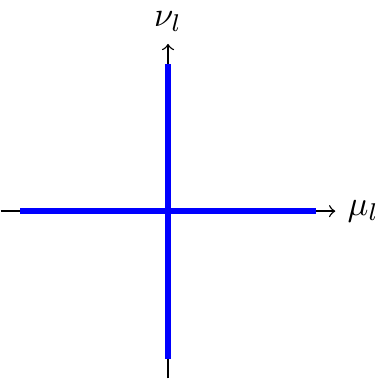}
\includegraphics[width=4.0cm]{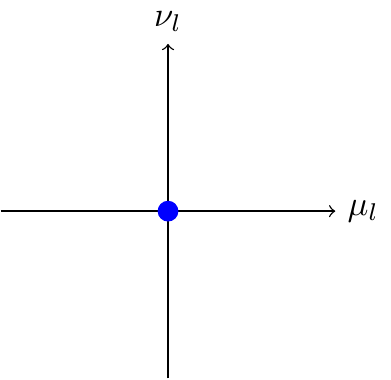}
\caption{Geometric illustrations of weak, M-, and S-stationarity for an index $l\in I^{GH}(\bar x)$.}
\label{fig:stationarities}
\end{figure}

In order to ensure that one of these stationarity notions plays the role of a necessary optimality condition for \eqref{eq:MPSC},
suitable problem-tailored constraint qualifications need to be valid. For the definition of such conditions, the following
so-called \emph{tightened nonlinear problem} is of interest:
\begin{equation}\label{eq:TNLP}\tag{TNLP}
	\begin{aligned}
		f(x)&\,\rightarrow\,\min\\
		g_i(x)&\,\leq\,0,&\qquad&i\in\mathcal M,&\\
		h_j(x)&\,=\,0,&\qquad&j\in\mathcal P,&\\
		G_l(x)&\,=\,0,&&l\in I^{G}(\bar x)\cup I^{GH}(\bar x),&\\
		H_l(x)&\,=\,0,&&l\in I^{H}(\bar x)\cup I^{GH}(\bar x).
	\end{aligned}
\end{equation}
Note that \eqref{eq:TNLP} is a standard nonlinear program.
\begin{definition}\label{def:MPSC_CQs}
	Let $\bar x\in X$ be a feasible point of \eqref{eq:MPSC}. Then \emph{MPSC-LICQ (MPSC-MFCQ)} is said to hold for \eqref{eq:MPSC} at
	$\bar x$ if LICQ (MFCQ) holds for \eqref{eq:TNLP} at $\bar x$, i.e., if the vectors
	\[
		\begin{split}
			\{\nabla g_i(\bar x)\,|\,i\in I^g(\bar x)\}\cup
				\Bigl[	&\{\nabla h_j(\bar x)\,|\,j\in\mathcal P\}\\
						&\cup\{\nabla G_l(\bar x)\,|\,l\in I^G(\bar x)\cup I^{GH}(\bar x)\}\\
						&\cup\{\nabla H_l(\bar x)\,|\,l\in I^H(\bar x)\cup I^{GH}(\bar x)\}\Bigr]
		\end{split}
	\]
	are linearly independent (positive-linearly independent).
\end{definition}

It is obvious that MPSC-LICQ is stronger than MPSC-MFCQ. 
Furthermore, MPSC-LICQ implies standard GCQ for \eqref{eq:MPSC} at the reference point.

In this paper, we will use another MPSC-tailored constraint qualification called MPSC-NNAMCQ, where NNAMCQ stands for the 
\emph{No Nonzero Abnormal Multiplier Constraint Qualification} which has been introduced to investigate optimization problems 
whose feasible sets are preimages of closed but not necessarily convex sets under continuously differentiable mappings, 
see \cite[Section~6.D]{RockafellarWets1998} and \cite{YeYe1997}. 
Clearly, \eqref{eq:MPSC} belongs to this problem class as well if one reformulates the switching constraints as
\begin{equation*}
    (G_l(x),H_l(x))\in C,\qquad l\in\mathcal Q,
\end{equation*}
with $C:=\{(a,b)\in\R^2\,|\,ab=0\}$.
\begin{definition}\label{def:MPSC_NNAMCQ}
	Let $\bar x\in X$ be a feasible point of \eqref{eq:MPSC}. Then \emph{MPSC-NNAMCQ} is said to hold for \eqref{eq:MPSC} at $\bar x$ if the following condition is valid:
	\[
		\left.
			\begin{aligned}
				&0=\sum\limits_{i\in I^g(\bar x)}\lambda_i\nabla g_i(\bar x)
					+\sum\limits_{j\in\mathcal P}\rho_j\nabla h_j(\bar x)\\
				&\qquad\qquad +\sum\limits_{l\in\mathcal Q}
					\Bigl[
						\mu_l\nabla G_l(\bar x)	+\nu_l\nabla H_l(\bar x)
					\Bigr]\\
				&\forall i\in I^{g}(\bar x)\colon\,\lambda_i\geq 0,\\
				&\forall l\in I^H(\bar x)\colon\,\mu_l=0,\\
				&\forall l\in I^{G}(\bar x)\colon\,\nu_l=0,\\
				&\forall l\in I^{GH}(\bar x)\colon\,\mu_l\nu_l=0
			\end{aligned}
		\right\}
		\,
		\Longrightarrow
		\,
		\lambda=0,\,\rho=0,\,\mu=0,\,\nu=0.
	\]
\end{definition}

It is easy to check that MPSC-NNAMCQ is implied by MPSC-MFCQ.
Note that a local minimizer of \eqref{eq:MPSC} where
MPSC-LICQ holds is an S-stationary point. 
Furthermore, one can easily check that the associated multipliers which solve the system of S-stationarity are
uniquely determined. Under MPSC-MFCQ (and, thus, MPSC-NNAMCQ), 
a local minimizer of \eqref{eq:MPSC} is, in general,
only M-stationary. Finally, there exist several problem-tailored constraint qualifications for \eqref{eq:MPSC} which are
weaker than MPSC-MFCQ but also imply that local solutions are M-stationary, see \cite{Mehlitz2018}.

\section{The relaxation scheme and its convergence properties}\label{sec:relaxation_scheme_KS}
\subsection{On the relaxation scheme}\label{sec:relaxation}

For our relaxation approach, we will make use of the function $\varphi\colon\R^2\to\R$ defined below:
\[
	\forall (a,b)\in\R^2\colon
	\quad
	\varphi(a,b):=
		\begin{cases}
			ab					&\text{if }a+b\geq 0,\\
			-\tfrac12(a^2+b^2)	&\text{if }a+b<0.
		\end{cases}
\]
The function $\varphi$ was introduced in \cite{KanzowSchwartz2013}
to study a relaxation method for the numerical solution of MPCCs. In the following lemma,
which parallels \cite[Lemma~3.1]{KanzowSchwartz2013}, some properties of $\varphi$ are
summarized.
\begin{lemma}\label{lem:properties_of_varphi}
	\begin{enumerate}
		\item[(a)] The function $\varphi$ is an NCP-function, i.e.\ it satisfies
		\[
			\forall (a,b)\in\R^2\colon
			\quad 
			\varphi(a,b)=0\,\Longleftrightarrow\,a\geq 0\,\land\,b\geq 0\,\land ab=0.
		\]
		\item[(b)] The function $\varphi$ is continuously differentiable and satisfies
		\[
			\forall (a,b)\in\R^2\colon
			\quad
			\nabla\varphi(a,b)=
				\begin{cases}
					\begin{pmatrix}
						b\\a
					\end{pmatrix}
						&\text{if }a+b\geq 0,\\
					\begin{pmatrix}
						-a\\-b
					\end{pmatrix}
						&\text{if }a+b<0.
				\end{cases}
		\]
	\end{enumerate}
\end{lemma}

For some parameter $t\geq 0$ as well as indices $s\in\mathcal S:=\{1,2,3,4\}$ and $l\in\mathcal Q$, 
we define functions $\Phi^s_l(\cdot;t)\colon\R^n\to\R$ via
\[
	\begin{aligned}
			\Phi^1_l(x;t)&:=\varphi(G_l(x)-t,H_l(x)-t),&\quad
			\Phi^2_l(x;t)&:=\varphi(-G_l(x)-t,H_l(x)-t),&\\
			\Phi^3_l(x;t)&:=\varphi(-G_l(x)-t,-H_l(x)-t),&\quad
			\Phi^4_l(x;t)&:=\varphi(G_l(x)-t,-H_l(x)-t)&
	\end{aligned}
\]
for any $x\in\R^n$.
Now, we are in position to introduce the surrogate problem of our interest:
\begin{equation}\label{eq:relaxedMPSC}\tag{P$(t)$}
	\begin{aligned}
		f(x)&\,\rightarrow\,\min&&&\\
		g_i(x)&\,\leq\,0&\qquad&i\in\mathcal M&\\
		h_j(x)&\,=\,0&&j\in\mathcal P&\\
		\Phi^s_l(x;t)&\,\leq\,0&&s\in\mathcal S,\,l\in\mathcal Q.
	\end{aligned}
\end{equation}
The feasible set of \eqref{eq:relaxedMPSC} will be denoted by $X(t)$. 
Figure \ref{fig:relaxed_feasible_set} provides an illustration of $X(t)$.

\begin{figure}[h]\centering
\includegraphics[width=4.0cm]{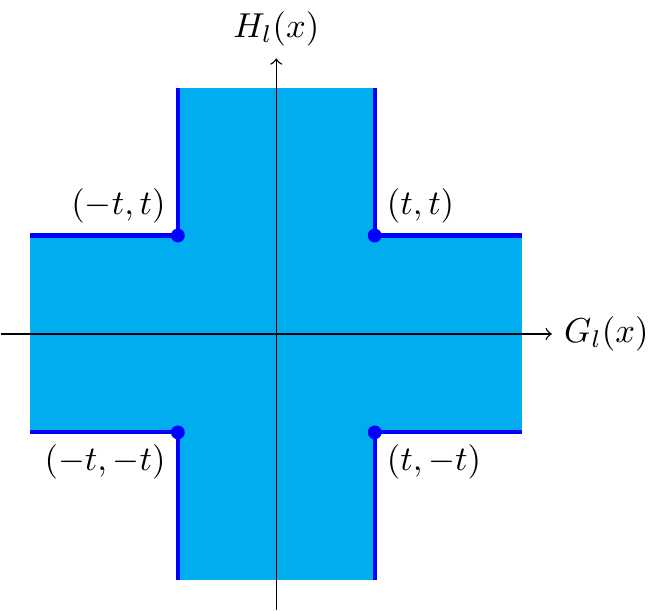}
\caption{Geometric interpretation of the relaxed feasible set $X(t)$.}
\label{fig:relaxed_feasible_set}
\end{figure}

Let $x\in X(t)$ be a feasible point of \eqref{eq:relaxedMPSC} for some fixed
parameter $t>0$. Later on, it will be beneficial to work with the index sets
defined below:
\begin{align*}
    I^{00}_{t,1}(x)&:=\{l\in\mathcal Q\,|\,G_l(x)=t,\,H_l(x)=t\},& 
        I^{00}_{t,2}(x)&:=\{l\in\mathcal Q\,|\,G_l(x)=-t,\,H_l(x)=t\},&\\
    I^{0+}_{t,1}(x)&:=\{l\in\mathcal Q\,|\,G_l(x)=t,\,H_l(x)>t\},&
        I^{0+}_{t,2}(x)&:=\{l\in\mathcal Q\,|\,G_l(x)=-t,\,H_l(x)>t\},&\\
    I^{+0}_{t,1}(x)&:=\{l\in\mathcal Q\,|\,G_l(x)>t,\,H_l(x)=t\},&
        I^{-0}_{t,2}(x)&:=\{l\in\mathcal Q\,|\,G_l(x)<-t,\,H_l(x)=t\},&\\
    I^{00}_{t,3}(x)&:=\{l\in\mathcal Q\,|\,G_l(x)=-t,\,H_l(x)=-t\},&
        I^{00}_{t,4}(x)&:=\{l\in\mathcal Q\,|\,G_l(x)=t,\,H_l(x)=-t\},&\\
    I^{0-}_{t,3}(x)&:=\{l\in\mathcal Q\,|\,G_l(x)=-t,\,H_l(x)<-t\},&
        I^{0-}_{t,4}(x)&:=\{l\in\mathcal Q\,|\,G_l(x)=t,\,H_l(x)<-t\},&\\
    I^{-0}_{t,3}(x)&:=\{l\in\mathcal Q\,|\,G_l(x)<-t,\,H_l(x)=-t\},&
        I^{+0}_{t,4}(x)&:=\{l\in\mathcal Q\,|\,G_l(x)>t,\,H_l(x)=-t\}.&
\end{align*}
Note that all these sets are pairwise disjoint. 
The index sets $I_{t,1}^{00}(x)$, $I_{t,1}^{0+}(x)$, and $I_{t,1}^{+0}(x)$ subsume the three possible cases 
where the constraints $\Phi_l^1(x;t)\le 0$ ($l\in \mathcal{Q}$) are active. 
Similarly, the other index sets cover those indices where the constraints $\Phi_l^s(x;t)\le 0$ ($l\in\mathcal{Q}$, $s\in\{2,3,4\}$) are active. 
It follows that an index $l\in\mathcal{Q}$ which does not belong to any of the above sets 
is inactive for \eqref{eq:relaxedMPSC} and can therefore be disregarded (locally).

In order to address any of the four quadrants separately, we will exploit
\begin{align*}
    I^0_{t,1}(x)&:=I^{00}_{t,1}(x)\cup I^{0+}_{t,1}(x)\cup I^{+0}_{t,1}(x),&
    I^0_{t,2}(x)&:=I^{00}_{t,2}(x)\cup I^{0+}_{t,2}(x)\cup I^{-0}_{t,2}(x),\\
    I^0_{t,3}(x)&:=I^{00}_{t,3}(x)\cup I^{0-}_{t,3}(x)\cup I^{-0}_{t,3}(x),&
    I^0_{t,4}(x)&:=I^{00}_{t,4}(x)\cup I^{0-}_{t,4}(x)\cup I^{+0}_{t,4}(x),
\end{align*}
i.e., for fixed $s\in\mathcal S$, $I_{t,s}^0(x)$ collects all indices $l\in\mathcal{Q}$ where the constraint $\Phi_l^s(x;t)\leq 0$ is active.

For brevity, we set
\begin{align*}
    I^{00}_t(x)&:=\bigcup_{s\in \mathcal S}I^{00}_{t,s}(x),\\
    I^{0\pm}_t(x)&:=I^{0+}_{t,1}(x)\cup I^{0+}_{t,2}(x)\cup I^{0-}_{t,3}(x)\cup I^{0-}_{t,4}(x),\\
    I^{\pm 0}_t(x)&:=I^{+0}_{t,1}(x)\cup I^{-0}_{t,2}(x)\cup I^{-0}_{t,3}(x)\cup I^{+0}_{t,4}(x).
\end{align*}
Thus, we collect all indices in $I^{0\pm}_t(x)$ where $G_l(x)\in\{-t,t\}$ holds
while $|H_l(x)|>t$ is valid. Similarly, $I^{\pm 0}_t(x)$ comprises all indices
where $H_l(x)\in\{-t,t\}$ and $|G_l(x)|>t$ hold true. The set $I^{00}_t(x)$
contains all those indices where $G_l(x),H_l(x)\in\{-t,t\}$ is valid.

If $\bar x\in X$ is feasible to \eqref{eq:MPSC}, $x$ lies in a sufficiently small
neighborhood of $\bar x$, and $t\geq 0$ is sufficiently small, then the following
inclusions follow from the continuity of all appearing functions:
\begin{equation}\label{eq:upper_estimate_index_sets}
	\begin{split}
		I^{00}_{t}(x)\cup I^{0\pm}_{t}(x)
		&\subset I^{G}(\bar x)\cup I^{GH}(\bar x),\\
		I^{00}_{t}(x)\cup I^{\pm 0}_t(x)
		&\subset I^{H}(\bar x)\cup I^{GH}(\bar x).
	\end{split}
\end{equation}
In the lemma below, we present explicit formulas for the gradients of $\Phi^s_l(\cdot;t)$ 
with $l\in\mathcal Q$ and $s\in\mathcal S$. They can be derived exploiting Lemma \ref{lem:properties_of_varphi}
as well as the chain rule.
\begin{lemma}\label{lem:derivative_of_Phi}
	For $x\in\R^n$, $t>0$, and $l\in\mathcal Q$, the following formulas are valid:
\begin{align*}
			\nabla_x\Phi^1_l(x;t)&=
				\begin{cases}
					(H_l(x)-t)\nabla G_l(x)+(G_l(x)-t)\nabla H_l(x)	
						&\text{if }G_l(x)+H_l(x)\geq 2t,\\
					-(G_l(x)-t)\nabla G_l(x)-(H_l(x)-t)\nabla H_l(x)
						&\text{if }G_l(x)+H_l(x)<2t,
				\end{cases}\\
			\nabla_x\Phi^2_l(x;t)&=
				\begin{cases}
					(t-H_l(x))\nabla G_l(x)-(G_l(x)+t)\nabla H_l(x)	
						&\text{if }-G_l(x)+H_l(x)\geq 2t,\\
					-(G_l(x)+t)\nabla G_l(x)-(H_l(x)-t)\nabla H_l(x)
						&\text{if }-G_l(x)+H_l(x)<2t,
				\end{cases}\\
			\nabla_x\Phi^3_l(x;t)&=
				\begin{cases}
					(H_l(x)+t)\nabla G_l(x)+(G_l(x)+t)\nabla H_l(x)	
						&\text{if }-G_l(x)-H_l(x)\geq 2t,\\
					-(G_l(x)+t)\nabla G_l(x)-(H_l(x)+t)\nabla H_l(x)
						&\text{if }-G_l(x)-H_l(x)<2t,
				\end{cases}\\
			\nabla_x\Phi^4_l(x;t)&=
				\begin{cases}
					-(H_l(x)+t)\nabla G_l(x)+(t-G_l(x))\nabla H_l(x)	
						&\text{if }G_l(x)-H_l(x)\geq 2t,\\
					-(G_l(x)-t)\nabla G_l(x)-(H_l(x)+t)\nabla H_l(x)
						&\text{if }G_l(x)-H_l(x)<2t.
				\end{cases}
\end{align*}
	Particularly, we have
\begin{align*}
			&\forall l\in I^{0}_{t,1}(x)\colon\quad&
				\nabla_x\Phi^1_l(x;t)&=
					\begin{cases}
						(G_l(x)-t)\nabla H_l(x) & \text{if }l\in I^{+0}_{t,1}(x),\\
						(H_l(x)-t)\nabla G_l(x) & \text{if }l\in I^{0+}_{t,1}(x),\\
						0						& \text{if }l\in I^{00}_{t,1}(x),
					\end{cases}&\\
			&\forall l\in I^{0}_{t,2}(x)\colon\quad&
				\nabla_x\Phi^2_l(x;t)&=
					\begin{cases}
						-(G_l(x)+t)\nabla H_l(x) & \text{if }l\in I^{-0}_{t,2}(x),\\
						(t-H_l(x))\nabla G_l(x) & \text{if }l\in I^{0+}_{t,2}(x),\\
						0						& \text{if }l\in I^{00}_{t,2}(x),
					\end{cases}&\\
			&\forall l\in I^{0}_{t,3}(x)\colon\quad&
				\nabla_x\Phi^3_l(x;t)&=
					\begin{cases}
						(G_l(x)+t)\nabla H_l(x) & \text{if }l\in I^{-0}_{t,3}(x),\\
						(H_l(x)+t)\nabla G_l(x) & \text{if }l\in I^{0-}_{t,3}(x),\\
						0						& \text{if }l\in I^{00}_{t,3}(x),
					\end{cases}&\\
			&\forall l\in I^{0}_{t,4}(x)\colon\quad&
				\nabla_x\Phi^4_l(x;t)&=
					\begin{cases}
						(t-G_l(x))\nabla H_l(x) & \text{if }l\in I^{+0}_{t,4}(x),\\
						-(H_l(x)+t)\nabla G_l(x) & \text{if }l\in I^{0-}_{t,4}(x),\\
						0						& \text{if }l\in I^{00}_{t,4}(x).
					\end{cases}&
\end{align*}
\end{lemma}

As a corollary of the above lemma, we obtain an explicit formula for the linearization cone
associated with \eqref{eq:relaxedMPSC}.
\begin{corollary}\label{cor:linearization_cone_relaxedMPSC}
	Fix $t>0$ and a feasible point $x\in X(t)$ of \eqref{eq:relaxedMPSC}.
	Then the following formula is valid:
	\[
		\mathcal L_{X(t)}(x)
		=
		\left\{
			d\in\R^n\,\middle|\,
			\begin{aligned}
				\nabla g_i(x)\cdot d	&\,\leq\,0&		&i\in I^g(x)\\
				\nabla h_j(x)\cdot d	&\,\leq\,0&		&j\in\mathcal P\\
				\nabla G_l(x)\cdot d	&\,\leq\,0&		&l\in I^{0+}_{t,1}(x)\cup I^{0-}_{t,4}(x)\\
				\nabla G_l(x)\cdot d	&\,\geq\,0&		&l\in I^{0+}_{t,2}(x)\cup I^{0-}_{t,3}(x)\\
				\nabla H_l(x)\cdot d	&\,\leq\,0&		&l\in I^{+0}_{t,1}(x)\cup I^{-0}_{t,2}(x)\\
				\nabla H_l(x)\cdot d	&\,\geq\,0&		&l\in I^{-0}_{t,3}(x)\cup I^{+0}_{t,4}(x)
			\end{aligned}
		\right\}.
	\]
\end{corollary}

The upcoming lemma justifies that \eqref{eq:relaxedMPSC} is indeed a \emph{relaxation} of \eqref{eq:MPSC}. Its
proof parallels the one of \cite[Lemma~3.2]{KanzowSchwartz2013}.
\begin{lemma}\label{lem:properties_of_relaxed_feasible_set}
	The following statements hold:
	\begin{enumerate}
		\item[\textup{(P1)}] $X(0)=X$,
		\item[\textup{(P2)}] $0\leq t_1\leq t_2\,\Longrightarrow\,X(t_1)\subset X(t_2)$,
		\item[\textup{(P3)}] $\bigcap_{t> 0}X(t)=X$.
	\end{enumerate}
\end{lemma}

Now, we are in position to characterize a conceptual method for the numerical solution of \eqref{eq:MPSC}: 
First, a sequence $\{t_k\}_{k\in\N}\subset\R_+$ of positive relaxation parameters is chosen which converges
to zero. Next, one solves the surrogate problem \hyperref[eq:relaxedMPSC]{(P$(t_k)$)} via standard methods.
If one computes a (local) minimizer of one of these surrogate problems which is feasible to \eqref{eq:MPSC},
then it is already a (local) minimizer of \eqref{eq:MPSC}. In general, it will be only possible to compute
KKT points of the surrogate problem. However, if such a sequence of KKT points converges to some point 
$\bar x\in\R^n$, then this point must be feasible to \eqref{eq:MPSC} by construction. Furthermore, it
will be shown in Section \ref{sec:convergence_properties} that whenever MPSC-NNAMCQ is valid at $\bar x$, then
it is an M-stationary point of \eqref{eq:MPSC}.

\subsection{Convergence properties}\label{sec:convergence_properties}

In this section, we analyze the theoretical properties of our relaxation scheme. 
In order to do so, we fix some standing assumptions below.
\begin{assumption}\label{ass:convergence_analysis}
	Let $\{t_k\}_{k\in\N}\subset\R_+$ be a sequence of positive relaxation parameters converging to zero.
	For each $k\in\N$, let $x_k\in X(t_k)$ be a KKT point of \hyperref[eq:relaxedMPSC]{\textup{(P$(t_k)$)}}. 
	We assume that $\{x_k\}_{k\in\N}$ converges to $\bar x\in \R^n$. 
	Note that $\bar x\in X$ holds by Lemma~\ref{lem:properties_of_relaxed_feasible_set}.
\end{assumption}

First, we will show that whenever MPSC-NNAMCQ is valid at $\bar x$, then it is an M-stationary point
of \eqref{eq:MPSC}. Second, it will be demonstrated that MPSC-LICQ at $\bar x$
implies that GCQ holds for the surrogate problem \eqref{eq:relaxedMPSC} at all feasible points
from a sufficiently small neighborhood of $\bar x$ and sufficiently small $t>0$. This property
ensures that local minima of the surrogate problem \eqref{eq:relaxedMPSC} which are located near $\bar x$
are in fact KKT points. This way, it is shown that Assumption \ref{ass:convergence_analysis} is 
reasonable.

\begin{theorem}\label{thm:convergence_to_M-stationary_point}
	Let Assumption \ref{ass:convergence_analysis} be valid. 
	Suppose that MPSC-NNAMCQ holds at $\bar x$. 
	Then $\bar x$ is an M-stationary point of \eqref{eq:MPSC}.
\end{theorem}
\begin{proof}
	Since $x_k$ is a KKT point of \hyperref[eq:relaxedMPSC]{(P$(t_k)$)}, there exist multipliers
	$\lambda^k\in\R^m$, $\rho^k\in\R^p$, and 
	$\alpha^k,\beta^k,\gamma^k,\delta^k\in\R^q$ which solve the following system:
	\[
		\begin{split}
			&0=\nabla f(x_k)
				+\sum\limits_{i\in I^g(x_k)}\lambda^k_i\nabla g_i(x_k)
				+\sum\limits_{j\in\mathcal P}\rho^k_j\nabla h_j(x_k)\\
			&\qquad\qquad
				+\sum\limits_{l\in\mathcal Q}
					\Bigl[
						\alpha^k_l\nabla_x\Phi^1_l(x_k;t_k)+\beta^k_l\nabla_x\Phi^2_l(x_k;t_k)
						+\gamma^k_l\nabla_x\Phi^3_l(x_k;t_k)+\delta^k_l\nabla_x\Phi^4_l(x_k;t_k)
					\Bigr],\\
			&\forall i\in I^g(x_k)\colon\,\lambda^k_i\geq 0;\,\forall i\in\mathcal M\setminus I^g(x_k)\colon\,\lambda^k_l=0,\\
			&\forall l\in I^{0}_{t_k,1}(x_k)\colon\,\alpha^k_l\geq 0;\,\forall l\in\mathcal Q\setminus I^{0}_{t_k,1}(x_k)\colon\,\alpha^k_l= 0,\\
			&\forall l\in I^{0}_{t_k,2}(x_k)\colon\,\beta^k_l\geq 0;\,\forall l\in\mathcal Q\setminus I^{0}_{t_k,2}(x_k)\colon\,\beta^k_l= 0,\\
			&\forall l\in I^{0}_{t_k,3}(x_k)\colon\,\gamma^k_l\geq 0;\,\forall l\in\mathcal Q\setminus I^{0}_{t_k,3}(x_k)\colon\,\gamma^k_l= 0,\\
			&\forall l\in I^{0}_{t_k,4}(x_k)\colon\,\delta^k_l\geq 0;\,\forall l\in\mathcal Q\setminus I^{0}_{t_k,4}(x_k)\colon\,\delta^k_l= 0.
		\end{split}
	\]
	Next, let us define new multipliers $\alpha_G^k,\alpha_H^k,\beta_G^k,\beta^k_H,\gamma^k_G,\gamma^k_H,\delta^k_G,\delta^k_H\in\R^q$ 
	as stated below for all $l\in\mathcal Q$:
\begin{align*}
			\alpha^k_{G,l}&:=
				\begin{cases}
					\alpha^k_l(H_l(x_k)-t_k)	&l\in I^{0+}_{t_k,1}(x_k),\\
					0							&\text{otherwise,}
				\end{cases}
				&
			\alpha^k_{H,l}&:=
				\begin{cases}
					\alpha^k_l(G_l(x_k)-t_k)	&l\in I^{+0}_{t_k,1}(x_k),\\
					0							&\text{otherwise,}
				\end{cases}
				&\\
			\beta^k_{G,l}&:=
				\begin{cases}
					\beta^k_l(t_k-H_l(x_k))		&l\in I^{0+}_{t_k,2}(x_k),\\
					0							&\text{otherwise,}
				\end{cases}
				&
			\beta^k_{H,l}&:=
				\begin{cases}
					\beta^k_l(-G_l(x_k)-t_k)	&l\in I^{-0}_{t_k,2}(x_k),\\
					0							&\text{otherwise,}
				\end{cases}
				&\\
			\gamma^k_{G,l}&:=
				\begin{cases}
					\gamma^k_l(H_l(x_k)+t_k)	&l\in I^{0-}_{t_k,3}(x_k),\\
					0							&\text{otherwise,}
				\end{cases}
				&
			\gamma^k_{H,l}&:=
				\begin{cases}
					\gamma^k_l(G_l(x_k)+t_k)	&l\in I^{-0}_{t_k,3}(x_k),\\
					0							&\text{otherwise,}
				\end{cases}
				&\\
			\delta^k_{G,l}&:=
				\begin{cases}
					\delta^k_l(-H_l(x_k)-t_k)	&l\in I^{0-}_{t_k,4}(x_k),\\
					0							&\text{otherwise,}
				\end{cases}
				&
			\delta^k_{H,l}&:=
				\begin{cases}
					\delta^k_l(t_k-G_l(x_k))	&l\in I^{+0}_{t_k,4}(x_k),\\
					0							&\text{otherwise.}
				\end{cases}
				&
\end{align*}
	Furthermore, we set $\mu^k:=\alpha^k_G+\beta^k_G+\gamma^k_G+\delta^k_G$ and 
	$\nu^k:=\alpha^k_H+\beta^k_H+\gamma^k_H+\delta^k_H$. 
	By definition, we have $\supp \mu^k\subset I^{0\pm}_{t_k}(x_k)$ 
	as well as $\supp\nu^k\subset I^{\pm 0}_{t_k}(x_k)$. 
	Thus, for sufficiently large $k\in\N$, \eqref{eq:upper_estimate_index_sets} yields
	\[
		\supp\mu^k\subset I^{G}(\bar x)\cup I^{GH}(\bar x),\qquad
		\supp\nu^k\subset I^{H}(\bar x)\cup I^{GH}(\bar x).
	\]
	Using Lemma \ref{lem:derivative_of_Phi}, 
	\begin{equation}\label{eq:KKT_new_multipliers}
		0=\nabla f(x_k)+\sum\limits_{i\in I^g(x_k)}\lambda^k_i\nabla g_i(x_k)
			+\sum\limits_{j\in\mathcal P}\rho^k_j\nabla h_j(x_k)\\
			+\sum\limits_{l\in\mathcal Q}
				\Bigl[
					\mu^k_l\nabla G_l(x_k)+\nu^k_l\nabla H_l(x_k)
				\Bigr]
	\end{equation}
	is obtained.
	
	Next, it will be shown that the sequence $\{(\lambda^k,\rho^k,\mu^k,\nu^k)\}_{k\in\N}$ is bounded.
	Assuming the contrary, we define
	\[
		\forall k\in\N\colon
		\quad 
		(\tilde\lambda^k,\tilde\rho^k,\tilde\mu^k,\tilde\nu^k)
		:=
		\frac{(\lambda^k,\rho^k,\mu^k,\nu^k)}{\norm{(\lambda^k,\rho^k,\mu^k,\nu^k)}{2}}.
	\]
	Clearly, $\{(\tilde\lambda^k,\tilde\rho^k,\tilde\mu^k,\tilde\nu^k)\}_{k\in\N}$ is bounded and, thus,
	possesses a converging subsequence (without relabeling) with nonvanishing limit 
	$(\tilde\lambda,\tilde\rho,\tilde\mu,\tilde\nu)$. The continuity of $g$ yields 
	$\supp\tilde\lambda\subset I^g(\bar x)$. The above considerations yield
	\[
		\supp\tilde\mu\subset I^{G}(\bar x)\cup I^{GH}(\bar x),
		\quad
		\supp \tilde\nu\subset I^{H}(\bar x)\cup I^{GH}(\bar x).
	\]
	Choose $l\in I^{GH}(\bar x)$ arbitrarily. If $\tilde\mu_l\neq 0$ holds true, then
	$\tilde\mu^k_l\neq 0$ must be valid for sufficiently large $k\in\N$. By definition of $\mu^k_l$,
	$l\in I^{0\pm}_{t_k}(x_k)$ follows. Since $I^{0\pm}_{t_k}(x_k)$ and $I^{\pm 0}_{t_k}(x_k)$
	are disjoint, $\nu^k_l=0$ holds for sufficiently large $k\in\N$. This yields $\tilde\nu^k_l=0$
	for sufficiently large $k\in\N$, i.e.\ $\tilde\nu_l=0$ is obtained. This shows
	\[\forall l\in I^{GH}(\bar x)\colon\quad \tilde\mu_l\tilde\nu_l=0.\]
	Dividing \eqref{eq:KKT_new_multipliers} by $\norm{(\lambda^k,\rho^k,\mu^k,\nu^k)}{2}$, taking the
	limit $k\to\infty$, respecting the continuous differentiability of $f$, $g$, $h$, $G$, as well as $H$,
	and invoking the above arguments, we obtain
	\[
		\begin{split}
			&0=\sum\limits_{i\in I^g(\bar x)}\tilde\lambda_i\nabla g_i(\bar x)
				+\sum\limits_{j\in\mathcal P}\tilde\rho_j\nabla h_j(\bar x)
				+\sum\limits_{l\in\mathcal Q}
					\Bigl[
						\tilde\mu_l\nabla G_l(\bar x)+\tilde\nu_l\nabla H_l(\bar x)
					\Bigr],\\
			&\forall i\in I^g(\bar x)\colon\,\tilde\lambda_i\geq 0;\,\forall i\in\mathcal M\setminus I^g(\bar x)\colon\,\tilde\lambda_i=0,\\
			&\forall l\in I^H(\bar x)\colon\,\tilde\mu_l=0,\\
			&\forall l\in I^G(\bar x)\colon\,\tilde\nu_l=0,\\
			&\forall l\in I^{GH}(\bar x)\colon\,\tilde\mu_l\tilde\nu_l=0.
		\end{split}
	\]
	Due to the fact that $(\tilde\lambda,\tilde\rho,\tilde\mu,\tilde\nu)$ does not vanish, 
	this is a contradiction to the validity of MPSC-NNAMCQ. 
	Thus, $\{(\lambda^k,\rho^k,\mu^k,\nu^k)\}_{k\in\N}$ is bounded.
	
	We assume w.l.o.g.\ that $\{(\lambda^k,\rho^k,\mu^k,\nu^k)\}_{k\in\N}$ converges to 
	$(\lambda,\rho,\mu,\nu)$ (otherwise, we choose an appropriate subsequence). Reprising
	the above arguments, we have
	 \[
	 	\begin{split}
	 		&\forall i\in I^g(\bar x)\colon\,\lambda_i\geq 0;\,\forall i\in\mathcal M\setminus I^g(\bar x)\colon\,\lambda_i=0,\\
			&\forall l\in I^H(\bar x)\colon\,\mu_l=0,\\
			&\forall l\in I^G(\bar x)\colon\,\nu_l=0,\\
			&\forall l\in I^{GH}(\bar x)\colon\,\mu_l\nu_l=0.
		\end{split}
	\]
	Taking the limit in \eqref{eq:KKT_new_multipliers} and respecting the continuous 
	differentiability of all appearing mappings, we obtain
	\[
		0=\nabla f(\bar x)+\sum\limits_{i\in I^g(\bar x)}\lambda_i\nabla g_i(\bar x)
			+\sum\limits_{j\in\mathcal P}\rho_j\nabla h_j(\bar x)\\
			+\sum\limits_{l\in\mathcal Q}
				\Bigl[
					\mu_l\nabla G_l(\bar x)+\nu_l\nabla H_l(\bar x)
				\Bigr].
	\]
	This shows that $\bar x$ is an M-stationary point of \eqref{eq:MPSC}.
\end{proof}

In order to show that the validity of MPSC-LICQ at the limit point $\bar x$ ensures that GCQ holds
at all feasible points of \eqref{eq:relaxedMPSC} sufficiently close to $\bar x$ where $t$ is sufficiently
small, we need to study the variational geometry of the sets $X(t)$ in some more detail. 

Fix $t>0$ and some point $\tilde x\in X(t)$. 
For an arbitrary index set $I\subset I^{00}_t(\tilde x)$, we consider the subsequent program:
\begin{equation}\label{eq:relaxedMPSC_decomposed}\tag{P$(t,\tilde x,I)$}
	\begin{aligned}
		f(x)&\,\rightarrow\,\min\\
		g_i(x)&\,\leq\,0,&\qquad&i\in\mathcal M,\\
		h_j(x)&\,=\,0,&&j\in\mathcal P,\\
		-t\,\leq\,G_l(x)&\,\leq\,t,&&l\in I^{0\pm}_t(\tilde x)\cup I,\\
		-t\,\leq\,H_l(x)&\,\leq\,t,&&l\in I^{\pm 0}_t(\tilde x)\cup (I^{00}_t(\tilde x)\setminus I),\\
		\Phi^s_l(x;t)&\,\leq\,0,&&l\in\mathcal Q\setminus I^0_{t,s}(\tilde x),\,s\in\mathcal S.&
	\end{aligned}
\end{equation}
The feasible set of \eqref{eq:relaxedMPSC_decomposed} will be denoted by $X(t,\tilde x,I)$. 
Clearly, $\tilde x$ is a feasible point of \eqref{eq:relaxedMPSC_decomposed} 
for arbitrary $I\subset I^{00}_t(\tilde x)$. Furthermore, $X(t,\tilde x,I)\subset X(t)$ is valid 
for any choice of $I\subset I^{00}_t(\tilde x)$.
\begin{lemma}\label{lem:decomposition_of_tangent_cone}
	For fixed $t>0$ and $\tilde x\in X(t)$, we have
	\[
		\mathcal T_{X(t)}(\tilde x)
		=
		\bigcup\limits_{I\subset I^{00}_{t}(\tilde x)}\mathcal T_{X(t,\tilde x,I)}(\tilde x).
	\]
\end{lemma}
\begin{proof}
We show both inclusions separately.\\
``$\subset$'' Fix an arbitrary direction $d\in\mathcal T_{X(t)}(\tilde x)$. Then we find sequences 
	$\{y_k\}_{k\in\N}\subset X(t)$ and $\{\tau_k\}_{k\in\N}\subset\R_+$ 
	such that $y_k\to \tilde{x}$, $\tau_k\downarrow 0$, and $(y_k-\tilde x)/\tau_k\to d$ as $k\to\infty$. 
	It is sufficient to verify the existence of an index set $\bar I\subset I^{00}_t(\tilde x)$ such
	that $\{y_k\}_{k\in\N}\cap X(t,\tilde x,\bar I)$ possesses infinite cardinality since this 
	already gives us $d\in\mathcal T_{X(t,\tilde x,\bar I)}(\tilde x)$.
	
	Fix $k\in\N$ sufficiently large and $l\in I^{0\pm}_{t}(\tilde x)$. Then, due to continuity
	of $G_l$, $H_l$, as well as $\varphi$ and feasibility of $y_k$ to \eqref{eq:relaxedMPSC}, 
	we either have $l\in I^{0\pm}_{t}(y_k)$ and, thus, $G_l(y_k)\in\{-t,t\}$, 
	or $-t< G_l(y_k)<t$. Similarly, we obtain $-t\leq H_l(y_k)\leq t$ for all $l\in I^{\pm0}_t(\tilde x)$.
	Due to feasibility of $y_k$ to \eqref{eq:relaxedMPSC},
	we have $-t\leq G_l(y_k)\leq t$ or $-t\leq H_l(y_k)\leq t$. Thus, setting
	\[
		I_k:=I^{00}_t(\tilde x)\cap\{l\in\mathcal Q\,|\,-t\leq G_l(y_k)\leq t\},
	\]
	$y_k\in X(t,\tilde x,I_k)$ is valid. Since there are only finitely many subsets of $I^{00}_t(\tilde x)$ while
	$\{y_k\}_{k\in\N}$ is infinite, there must exist $\bar I\subset I^{00}_t(\tilde x)$ such that
	$\{y_k\}_{k\in\N}\cap X(t,\tilde x,\bar I)$ is of infinite cardinality.\\
``$\supset$'' By definition of the tangent cone, we easily obtain 
	$\mathcal T_{X(t,\tilde x,I)}(\tilde x)\subset\mathcal T_{X(t)}(\tilde x)$
	for any $I\subset I^{00}_t(\tilde x)$. Taking the union over all subsets of $I^{00}_t(\tilde x)$
	yields the desired inclusion.
\end{proof}

\begin{theorem}\label{thm:MPSC_LICQ_yields_GCQ}
	Let $\bar x\in X$ be a feasible point of \eqref{eq:MPSC} where MPSC-LICQ is satisfied.
	Then there exist $\bar t>0$ and a neighborhood $U\subset\R^n$ of $\bar x$ such that GCQ holds
	for \eqref{eq:relaxedMPSC} at all points from $X(t)\cap U$ for all $t\in (0,\bar{t}]$.
\end{theorem}
\begin{proof}
	Due to the validity of MPSC-LICQ at $\bar x$ and the continuous differentiability
	of $g$, $h$, $G$, and $H$, the gradients
	\[
		\begin{split}
			\{\nabla g_i(x)\,|\,i\in I^g(\bar x)\}&\cup\{\nabla h_j(x)\,|\,j\in\mathcal P\}\\
			&\cup\{\nabla G_l(x)\,|\,l\in I^G(\bar x)\cup I^{GH}(\bar x)\}\\
			&\cup\{\nabla H_l(x)\,|\,l\in I^H(\bar x)\cup I^{GH}(\bar x)\}
		\end{split}
	\]
	are linearly independent for all $x$ which are chosen from a sufficiently small
	neighborhood $\mathcal V$ of $\bar x$, 
	see Lemma \ref{lem:local_stability_of_positive_linear_independence}.
	Invoking \eqref{eq:upper_estimate_index_sets}, we can choose a neighborhood
	$U\subset\mathcal V$ of $\bar x$ and $\bar t>0$ such that for any $\tilde x\in X(t)\cap U$, 
	where $t\in (0,\bar{t}]$ holds, 
	we have 
\begin{equation}\label{eq:ThmMPSCLICQyieldsGCQ1}
		\begin{split}
			I^g(\tilde x)&\subset I^g(\bar x),\\
			I^{00}_t(\tilde x)\cup I^{0\pm}_t(\tilde x)&\subset I^G(\bar x)\cup I^{GH}(\bar x),\\
			I^{00}_t(\tilde x)\cup I^{\pm 0}_t(\tilde x)&\subset I^{H}(\bar x)\cup I^{GH}(\bar x).
		\end{split}
\end{equation}
	Particularly, for any such $\tilde x\in\R^n$ and $I\subset I^{00}_t(\tilde x)$, the gradients
	\[
		\begin{split}
			\{\nabla g_i(\tilde x)\,|\,i\in I^g(\tilde x)\}&\cup\{\nabla h_j(\tilde x)\,|\,j\in\mathcal P\}\\
			&\cup\{\nabla G_l(\tilde x)\,|\,l\in I^{0\pm}_t(\tilde x)\cup I\}\\
			&\cup\{\nabla H_l(\tilde x)\,|\,l\in I^{\pm 0}_t(\tilde x)\cup (I^{00}_t(\tilde x)\setminus I)\}
		\end{split}
	\]
	are linearly independent, i.e., standard LICQ is valid for \eqref{eq:relaxedMPSC_decomposed} at $\tilde x$
	for any set $I\subset I^{00}_t(\tilde x)$. This implies 
	$\mathcal T_{X(t,\tilde x,I)}(\tilde x)=\mathcal L_{X(t,\tilde x,I)}(\tilde x)$
	for any $I\subset I^{00}_t(\tilde x)$. Exploiting Lemma \ref{lem:decomposition_of_tangent_cone}, we obtain
	\[
		\mathcal T_{X(t)}(\tilde x)
		=
		\bigcup\limits_{I\subset I^{00}_t(\tilde x)}\mathcal L_{X(t,\tilde x,I)}(\tilde x).
	\]
	Computing the polar cone on both sides yields
	\begin{equation}\label{eq:tangents_via_intersection_of_linearized_tangents}
		\mathcal T_{X(t)}(\tilde x)^\circ
		=
		\bigcap\limits_{I\subset I^{00}_t(\tilde x)}\mathcal L_{X(t,\tilde x,I)}(\tilde x)^\circ.
	\end{equation}
	Define
	\[
		\begin{split}
			\mathcal I^+_G(t,\tilde x,I)&:=I^{0+}_{t,1}(\tilde x)\cup I^{0-}_{t,4}(\tilde x)
					\cup\bigl[I\cap \bigl(I^{00}_{t,1}(\tilde x)\cup I^{00}_{t,4}(\tilde x)\bigr)\bigr],\\
			\mathcal I^-_G(t,\tilde x,I)&:=I^{0+}_{t,2}(\tilde x)\cup I^{0-}_{t,3}(\tilde x)
					\cup\bigl[I\cap \bigl(I^{00}_{t,2}(\tilde x)\cup I^{00}_{t,3}(\tilde x)\bigr)\bigr],\\
			\mathcal I^+_H(t,\tilde x,I)&:=I^{+0}_{t,1}(\tilde x)\cup I^{-0}_{t,2}(\tilde x)
					\cup\bigl[\bigl(I^{00}_t(\tilde x)\setminus I\bigr)
						\cap \bigl(I^{00}_{t,1}(\tilde x)\cup I^{00}_{t,2}(\tilde x)\bigr)\bigr],\\
			\mathcal I^-_H(t,\tilde x,I)&:=I^{-0}_{t,3}(\tilde x)\cup I^{+0}_{t,4}(\tilde x)
					\cup\bigl[\bigl(I^{00}_t(\tilde x)\setminus I\bigr)
						\cap \bigl(I^{00}_{t,3}(\tilde x)\cup I^{00}_{t,4}(\tilde x)\bigr)\bigr]
		\end{split}
	\]
	and observe that these sets characterize the indices $l\in\mathcal{Q}$ 
	where the constraints $G_l(\tilde{x})\le t$, $G_l(\tilde{x})\ge -t$, $H_l(\tilde{x})\le t$, 
	and $H_l(\tilde{x})\ge -t$, respectively, are active in \eqref{eq:relaxedMPSC_decomposed}. 
	We therefore obtain
	\[
		\mathcal L_{X(t,\tilde x,I)}(\tilde x)
		=
		\left\{
			d\in\R^n\,\middle|\,
				\begin{aligned}
					\nabla g_i(\tilde x)\cdot d&\,\leq\,0&&i\in I^g(\tilde x)\\
					\nabla h_j(\tilde x)\cdot d&\,=\,0&&j\in\mathcal P\\
					\nabla G_l(\tilde x)\cdot d&\,\leq\,0&&l\in \mathcal I^+_G(t,\tilde x,I)\\
					\nabla G_l(\tilde x)\cdot d&\,\geq\,0&&l\in \mathcal I^-_G(t,\tilde x,I)\\
					\nabla H_l(\tilde x)\cdot d&\,\leq\,0&&l\in \mathcal I^+_H(t,\tilde x,I)\\
					\nabla H_l(\tilde x)\cdot d&\,\geq\,0&&l\in \mathcal I^-_H(t,\tilde x,I)
				\end{aligned}
		\right\}.
	\]
	Exploiting Lemma \ref{lem:polar_of_polyhedral_cone}, the polar of this cone is easily computed:
	\begin{equation}\label{eq:polar_of_decomposed_linearization_cone}
		\mathcal L_{X(t,\tilde x,I)}(\tilde x)^\circ
		=
		\left\{
			\eta\in\R^n\,\middle|\,
				\begin{aligned}
					&\eta=\sum\limits_{i\in I^g(\tilde x)}\lambda_i\nabla g_i(\tilde x)
						+\sum\limits_{j\in\mathcal P}\rho_j\nabla h_j(\tilde x)\\
					&\qquad\qquad
						+\sum\limits_{l\in\mathcal Q}
							\Bigl[
								\mu_l\nabla G_l(\tilde x)+\nu_l\nabla H_l(\tilde x)
							\Bigr]\\
					&\forall i\in I^g(\tilde x)\colon\,\lambda_i\geq 0\\
					&\forall l\in \mathcal Q\colon\,\mu_l
						\begin{cases}
							\geq 0	&\text{if }l\in \mathcal I^+_G(t,\tilde x,I),\\
							\leq 0	&\text{if }l\in \mathcal I^-_G(t,\tilde x,I),\\
							=0		&\text{otherwise},
						\end{cases}\\
					&\forall l\in \mathcal Q\colon\,\nu_l
						\begin{cases}
							\geq 0	&\text{if }l\in \mathcal I^+_H(t,\tilde x,I),\\
							\leq 0	&\text{if }l\in \mathcal I^-_H(t,\tilde x,I),\\
							=0		&\text{otherwise}
						\end{cases}
				\end{aligned}
		\right\}.
	\end{equation}
	
	Observing that \eqref{eq:relaxedMPSC} is a standard nonlinear problem, 
	$\mathcal T_{X(t)}(\tilde x)\subset\mathcal L_{X(t)}(\tilde x)$
	and, thus,
	$\mathcal L_{X(t)}(\tilde x)^\circ\subset \mathcal T_{X(t)}(\tilde x)^\circ$
	are inherent. It remains to show 
	$\mathcal T_{X(t)}(\tilde x)^\circ\subset \mathcal L_{X(t)}(\tilde x)^\circ$.
	Thus, choose $\eta\in\mathcal T_{X(t)}(\tilde x)^\circ$ arbitrarily. 
	Then, in particular, \eqref{eq:tangents_via_intersection_of_linearized_tangents} yields
	\[
		\eta\in
			\mathcal L_{X(t,\tilde x,\varnothing)}(\tilde x)^\circ
			\cap
			\mathcal L_{X(t,\tilde x,I^{00}_t(\tilde x))}(\tilde x)^\circ.
	\]
	Exploiting the representation \eqref{eq:polar_of_decomposed_linearization_cone},
	we find $\lambda_i,\lambda'_i\geq 0$ ($i\in I^g(\tilde x)$), $\rho,\rho'\in\R^p$,
	$\mu,\mu'\in\R^q$, as well as $\nu,\nu'\in\R^q$ which satisfy
	\[
		\begin{split}
			\eta&=\sum\limits_{i\in I^g(\tilde x)}\lambda_i\nabla g_i(\tilde x)
					+\sum\limits_{j\in\mathcal P}\rho_j\nabla h_j(\tilde x)
					+\sum\limits_{l\in\mathcal Q}
						\Bigl[
							\mu_l\nabla G_l(\tilde x)+\nu_l\nabla H_l(\tilde x)
						\Bigr]\\
				&=\sum\limits_{i\in I^g(\tilde x)}\lambda'_i\nabla g_i(\tilde x)
					+\sum\limits_{j\in\mathcal P}\rho'_j\nabla h_j(\tilde x)
					+\sum\limits_{l\in\mathcal Q}
						\Bigl[
							\mu'_l\nabla G_l(\tilde x)+\nu'_l\nabla H_l(\tilde x)
						\Bigr]	
		\end{split}
	\]
	and
\begin{align*}
			\mu_l&
				\begin{cases}
					\geq 0	&\text{if }l\in I^{0+}_{t,1}(\tilde x)\cup I^{0-}_{t,4}(\tilde x),\\
					\leq 0	&\text{if }l\in I^{0+}_{t,2}(\tilde x)\cup I^{0-}_{t,3}(\tilde x),\\
					=0		&\text{otherwise,}
				\end{cases}\\
			\mu'_l&
				\begin{cases}
					\geq 0	&\text{if }l\in I^{0+}_{t,1}(\tilde x)\cup I^{0-}_{t,4}(\tilde x)\cup I^{00}_{t,1}(\tilde x)\cup I^{00}_{t,4}(\tilde x),\\
					\leq 0	&\text{if }l\in I^{0+}_{t,2}(\tilde x)\cup I^{0-}_{t,3}(\tilde x)\cup I^{00}_{t,2}(\tilde x)\cup I^{00}_{t,3}(\tilde x),\\
					=0		&\text{otherwise,}
				\end{cases}\\
			\nu_l&
				\begin{cases}
					\geq 0	&\text{if }l\in I^{+0}_{t,1}(\tilde x)\cup I^{-0}_{t,2}(\tilde x)\cup I^{00}_{t,1}(\tilde x)\cup I^{00}_{t,2}(\tilde x),\\
					\leq 0	&\text{if }l\in I^{-0}_{t,3}(\tilde x)\cup I^{+0}_{t,4}(\tilde x)\cup I^{00}_{t,3}(\tilde x)\cup I^{00}_{t,4}(\tilde x),\\
					=0		&\text{otherwise,}
				\end{cases}\\
			\nu'_l&
				\begin{cases}
					\geq 0	&\text{if }l\in I^{+0}_{t,1}(\tilde x)\cup I^{-0}_{t,2}(\tilde x),\\
					\leq 0	&\text{if }l\in I^{-0}_{t,3}(\tilde x)\cup I^{+0}_{t,4}(\tilde x),\\
					=0		&\text{otherwise}
				\end{cases}
\end{align*}
	for all $l\in\mathcal Q$. Thus, we obtain
	\[
		\begin{split}
			0	&=\sum\limits_{i\in I^g(\tilde x)}(\lambda_i-\lambda_i')\nabla g_i(\tilde x)
					+\sum\limits_{j\in\mathcal P}(\rho_j-\rho_j')\nabla h_j(\tilde x)\\
				&\qquad\qquad
					+\sum\limits_{l\in\mathcal Q}
						\Bigl[
							(\mu_l-\mu_l')\nabla G_l(\tilde x)+(\nu_l-\nu_l')\nabla H_l(\tilde x)
						\Bigr].
		\end{split}
	\]
	Observing $\supp(\mu-\mu')\subset I^{00}_t(\tilde x)\cup I^{0\pm}_t(\tilde x)$ as well as 
	$\supp(\nu-\nu')\subset I^{00}_t(\tilde x)\cup I^{\pm 0}_t(\tilde x)$ and using \eqref{eq:ThmMPSCLICQyieldsGCQ1}, 
	we obtain $\lambda_i=\lambda_i'$ ($i\in I^g(\tilde x)$),
	$\rho=\rho'$, $\mu=\mu'$, as well as $\nu=\nu'$. Particularly,
	\[
		\eta=\sum\limits_{i\in I^g(\tilde x)}\lambda_i\nabla g_i(\tilde x)
					+\sum\limits_{j\in\mathcal P}\rho_j\nabla h_j(\tilde x)
					+\sum\limits_{l\in\mathcal Q}
						\Bigl[
							\mu_l\nabla G_l(\tilde x)+\nu_l'\nabla H_l(\tilde x)
						\Bigr].
	\]
	is obtained where $\supp\mu\subset I^{0\pm}_t(\tilde x)$ and $\supp\nu'\subset I^{\pm 0}_t(\tilde x)$
	hold true. Finally, we exploit Corollary \ref{cor:linearization_cone_relaxedMPSC} in order to
	see $\eta\in\mathcal L_{X(t)}(\tilde x)^\circ$. This shows the validity of the inclusion
	$\mathcal T_{X(t)}(\tilde x)^\circ\subset \mathcal L_{X(t)}(\tilde x)^\circ$ and, thereby, 
	GCQ holds true for \eqref{eq:relaxedMPSC} at $\tilde x$. Since $t\in(0,\bar t]$ and 
	$\tilde x\in X(t)\cap U$ were arbitrarily chosen, the proof is completed.
\end{proof}

\section{Remarks on other possible relaxation schemes}\label{sec:other_relaxations}

In this section, we discuss three more relaxation approaches for the numerical treatment
of \eqref{eq:MPSC} which are inspired by the rich theory on MPCCs. 
Particularly, the relaxation schemes of \cite{Scholtes2001}, \cite{SteffensenUlbrich2010}, 
as well as \cite{KadraniDussaultBenchakroun2009}, are adapted to the setting
of switching-constrained optimization.

\subsection{The relaxation scheme of Scholtes}\label{sec:Scholtes}

For some parameter $t\geq 0$, let us consider the surrogate problem
\begin{equation}\label{eq:relaxedMPSC_Scholtes}\tag{P$_\text{S}(t)$}
	\begin{aligned}
		f(x)&\,\rightarrow\,\min&&&\\
		g_i(x)&\,\leq\,0,&\qquad&i\in\mathcal M,&\\
		h_j(x)&\,=\,0,&&j\in\mathcal P,&\\
		-t\,\leq\,G_l(x)H_l(x)&\,\leq\,t,&&l\in\mathcal Q.&
	\end{aligned}
\end{equation}
This idea is inspired by Scholtes' global relaxation method
which was designed for the computational treatment of MPCCs, 
see \cite{Scholtes2001} and \cite[Section~3.1]{HoheiselKanzowSchwartz2013}.
The feasible set of \eqref{eq:relaxedMPSC_Scholtes} is denoted by $X_\text S(t)$
and visualized in Figure \ref{fig:relaxed_feasible_set_Scholtes}. 
Note that the family
$\{X_\text S(t)\}_{t\geq 0}$ possesses the same properties as the family
$\{X(t)\}_{t\geq 0}$ described in Lemma \ref{lem:properties_of_relaxed_feasible_set}.
Thus, Scholtes' relaxation is reasonable for switching-constrained problems
as well. In contrast to \eqref{eq:relaxedMPSC}, where we need four inequality
constraints in order to replace one original switching constraint,
one only needs two inequality constraints in \eqref{eq:relaxedMPSC_Scholtes}
for the same purpose. This is a significant advantage of \eqref{eq:relaxedMPSC_Scholtes}
over the surrogate \eqref{eq:relaxedMPSC}.
\begin{figure}[h]\centering
\includegraphics[width=4.0cm]{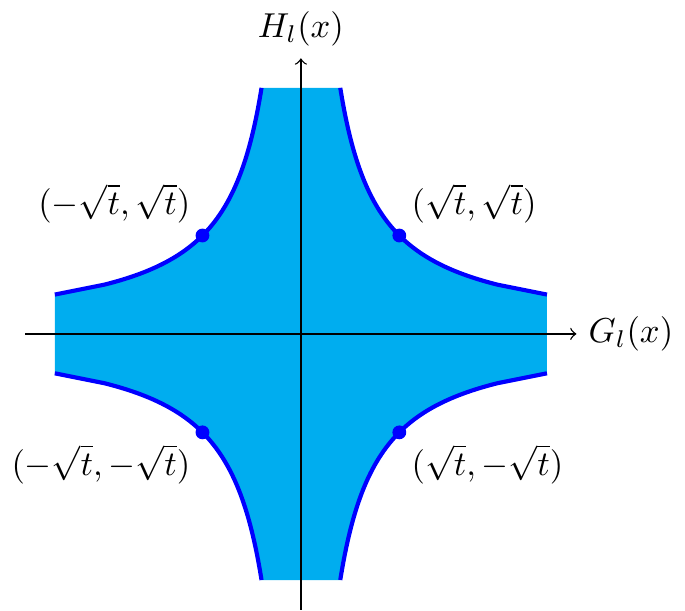}
\caption{Geometric interpretation of the relaxed feasible set $X_\text{S}(t)$.}
\label{fig:relaxed_feasible_set_Scholtes}
\end{figure}

It is well known from \cite[Theorem~3.1]{HoheiselKanzowSchwartz2013} that
Scholtes' relaxation approach finds Clarke-stationary points of MPCCs under
an MPCC-tailored version of MFCQ. Note that, in the
context of MPCCs, Clarke-stationarity is stronger than weak stationarity but weaker than 
Mordukhovich-stationarity.
Below, we want to generalize the result from \cite{HoheiselKanzowSchwartz2013}
to the problem \eqref{eq:MPSC}.

For the fixed parameter $t>0$ and a feasible point $x\in X_\text S(t)$
of \eqref{eq:relaxedMPSC_Scholtes}, 
we introduce the index sets
\[
	\begin{split}
		I^+_t(x)&:=\{l\in\mathcal Q\,|\,G_l(x)H_l(x)=t\},\\
		I^-_t(x)&:=\{l\in\mathcal Q\,|\,G_l(x)H_l(x)=-t\}.
	\end{split}
\]
In the upcoming theorem, we provide a convergence result of Scholtes' relaxation 
scheme for the problem \eqref{eq:MPSC}.
\begin{theorem}\label{thm:convergence_Scholtes}
	Let $\{t_k\}_{k\in\N}\subset\R_+$ be a sequence of positive relaxation parameters 
	converging to zero.
	For each $k\in\N$, let $x_k\in X_{\textup{S}}(t_k)$ be a KKT point of \hyperref[eq:relaxedMPSC_Scholtes]{\textup{(P$_{\text{S}}(t_k)$)}}.
	Assume that the sequence $\{x_k\}_{k\in\N}$ converges to a point $\bar x\in X$
	where MPSC-MFCQ holds. Then $\bar x$ is a W-stationary point of \eqref{eq:MPSC}.
\end{theorem}
\begin{proof}
	Noting that $x_k$ is a KKT point of \hyperref[eq:relaxedMPSC_Scholtes]{(P$_{\text{S}}(t_k)$)}, we
	find multipliers $\lambda^k\in\R^m$, $\rho^k\in\R^p$, and $\xi^k\in\R^q$ which
	satisfy the following conditions:
\begin{align*}
			&0=\nabla f(x_k)+\sum\limits_{i\in I^g(x_k)}\lambda^k_i\nabla g_i(x_k)
				+\sum\limits_{j\in\mathcal P}\rho^k_j\nabla h_j(x_k)\\
			&\qquad\qquad +\sum\limits_{l\in\mathcal Q}\xi^k_l
				\Bigl[
					H_l(x_k)\nabla G_l(x_k)+G_l(x_k)\nabla H_l(x_k)	
				\Bigr],\\
			&\forall i\in I^g(x_k)\colon\,\lambda^k_i\geq 0,\,
				\forall i\in\mathcal M\setminus I^g(x_k)\colon\,\lambda^k_i=0,\\
			&\forall l\in I^+_t(x_k)\colon\,\xi^k_l\geq 0,\\
			&\forall l\in I^-_t(x_k)\colon\,\xi^k_l\leq 0,\\
			&\forall l\in \mathcal Q\setminus(I^+_t(x_k)\cup I^-_t(x_k))\colon\,\xi^k_l=0.
\end{align*}
	For any $k\in\N$ and $l\in\mathcal Q$, let us define artificial multipliers 
	$\mu^k_l,\nu^k_l\in\R$ as stated below:
	\[
		\mu^k_l:=	\begin{cases}
						\xi^k_lH_l(x_k)	&l\in I^{G}(\bar x)\cup I^{GH}(\bar x),\\
						0				&l\in I^{H}(\bar x),
					\end{cases}
		\quad
		\nu^k_l:=	\begin{cases}
						\xi^k_lG_l(x_k)	&l\in I^{H}(\bar x)\cup I^{GH}(\bar x),\\
						0				&l\in I^{G}(\bar x).
					\end{cases}
	\]
	Thus, we obtain
	\begin{equation}\label{eq:Scholtes_stationarity_transformed}
		\begin{split}
			&0=\nabla f(x_k)+\sum\limits_{i\in I^g(x_k)}\lambda^k_i\nabla g_i(x_k)
				+\sum\limits_{j\in\mathcal P}\rho^k_j\nabla h_j(x_k)
				+\sum\limits_{l\in\mathcal Q}
				\Bigl[
					\mu^k_l\nabla G_l(x_k)+\nu^k_l\nabla H_l(x_k)
				\Bigr]\\
			&\qquad\qquad
				+\sum\limits_{l\in I^H(\bar x)}\xi^k_lH_l(x_k)\nabla G_l(x_k)
				+\sum\limits_{l\in I^G(\bar x)}\xi^k_lG_l(x_k)\nabla H_l(x_k).
		\end{split}
	\end{equation}
	Next, we are going to show that the sequence 
	$\{(\lambda^k,\rho^k,\mu^k,\nu^k,\xi^k_I)\}_{k\in\N}$
	is bounded where we used $I:=I^G(\bar x)\cup I^H(\bar x)$ for brevity. 
	We assume on the contrary that this is not the case and define
	\[
		\forall k\in\N\colon
		\quad
		(\tilde\lambda^k,\tilde\rho^k,\tilde\mu^k,\tilde\nu^k,\tilde\xi^k_I)
		:=
		\frac{ 	
				(\lambda^k,\rho^k,\mu^k,\nu^k,\xi^k_I)
			}
			{
				\norm{(\lambda^k,\rho^k,\mu^k,\nu^k,\xi^k_I)}{2}
			}.
	\]
	Clearly, $\{(\tilde\lambda^k,\tilde\rho^k,\tilde\mu^k,\tilde\nu^k,\tilde\xi^k_I)\}_{k\in\N}$
	is bounded and, thus, converges w.l.o.g.\ to a nonvanishing vector 
	$(\tilde\lambda,\tilde\rho,\tilde\mu,\tilde\nu,\tilde\xi_I)$ 
	(otherwise, a suitable subsequence is chosen). The continuity of $g$ ensures
	that $I^g(x_k)\subset I^g(\bar x)$ is valid for sufficiently large $k\in\N$.
	Dividing \eqref{eq:Scholtes_stationarity_transformed} by
	$\norm{(\lambda^k,\rho^k,\mu^k,\nu^k,\xi^k_I)}{2}$ and taking the limit
	$k\to\infty$ while respecting the continuous differentiability of all involved
	functions, we come up with
	\[
		\begin{split}
			&0=\sum\limits_{i\in I^g(\bar x)}\tilde\lambda_i\nabla g_i(\bar x)
				+\sum\limits_{j\in\mathcal P}\tilde\rho_j\nabla h_j(\bar x)
				+\sum\limits_{l\in\mathcal Q}
					\Bigl[
						\tilde\mu_l\nabla G_l(\bar x)+\tilde\nu_l\nabla H_l(\bar x)
					\Bigr],\\
			&\forall i\in I^g(\bar x)\colon\,\tilde\lambda_i\geq 0,\,
				\forall i\in\mathcal M\setminus I^g(\bar x)\colon\,\tilde\lambda_i=0,\\
			&\forall l\in I^H(\bar x)\colon\,\tilde\mu_l=0,\\
			&\forall l\in I^G(\bar x)\colon\,\tilde\nu_l=0.
		\end{split}
	\]
	Now, the validity of MPSC-MFCQ yields $\tilde\lambda=0$, $\tilde\rho=0$, $\tilde\mu=0$,
	and $\tilde\nu=0$. Hence, $\tilde\xi_{l_0}\neq 0$ holds for at least one index $l_0\in I$.
	Let us assume $l_0\in I^H(\bar x)$. Then we have $\nu^k_{l_0}=\xi^k_{l_0}G_{l_0}(x_k)$,
	which leads to
	\[
		\tilde\nu_{l_0}
			=\lim\limits_{k\to\infty}\frac{\nu^k_{l_0}}{\norm{(\lambda^k,\rho^k,\mu^k,\nu^k,\xi^k_I)}{2}}
			=\lim\limits_{k\to\infty}\frac{\xi^k_{l_0}G_{l_0}(x_k)}{\norm{(\lambda^k,\rho^k,\mu^k,\nu^k,\xi^k_I)}{2}}
			=\tilde\xi_{l_0}G_{l_0}(\bar x)\neq 0.
	\]
	This, however, is a contradiction since $\tilde\nu$ vanishes due to the above arguments.
	Similarly, the case $l_0\in I^G(\bar x)$ leads to a contradiction. 
	As a consequence, the sequence $\{(\lambda^k,\rho^k,\mu^k,\nu^k,\xi^k_I)\}_{k\in\N}$ is bounded.
	
	Thus, we may assume w.l.o.g.\ that this sequence converges to $(\lambda,\rho,\mu,\nu,\xi_I)$. Again,
	we take the limit in \eqref{eq:Scholtes_stationarity_transformed} and obtain
	\[	
		\begin{split}
			&0=\nabla f(\bar x)+\sum\limits_{i\in I^g(\bar x)}\lambda_i\nabla g_i(\bar x)
				+\sum\limits_{j\in\mathcal P}\rho_j\nabla h_j(\bar x)
				+\sum\limits_{l\in\mathcal Q}
					\Bigl[
						\mu_l\nabla G_l(\bar x)+\nu_l\nabla H_l(\bar x)
					\Bigr],\\
			&\forall i\in I^g(\bar x)\colon\,\lambda_i\geq 0,\,
				\forall i\in\mathcal M\setminus I^g(\bar x)\colon\,\lambda_i=0,\\
			&\forall l\in I^H(\bar x)\colon\,\mu_l=0,\\
			&\forall l\in I^G(\bar x)\colon\,\nu_l=0
		\end{split}
	\]
	which shows that $\bar x$ is a W-stationary point of \eqref{eq:MPSC}.
\end{proof}

Noting that no suitable definition of Clarke-stationarity 
(in particular, a \emph{reasonable} stationarity concept which is stronger than
W- but weaker than M-stationarity) seems to be 
available for \eqref{eq:MPSC}, Theorem \ref{thm:convergence_Scholtes} 
does not seem to be too surprising at all. 
The following example shows that we cannot expect any stronger results in general.
Thus, the qualitative properties of Scholtes' relaxation method are substantially 
weaker than those of the relaxation scheme proposed in Section \ref{sec:relaxation}.
\begin{example}\label{ex:Scholtes_relaxation}
	Let us consider the switching-constrained optimization problem
	\[
		\begin{split}
			\tfrac12(x_1-1)^2+\tfrac12(x_2-1)^2&\,\rightarrow\,\min\\
			x_1x_2&\,=\,0.
		\end{split}
	\]
	The globally optimal solutions of this program are given by $(1,0)$ as well as
	$(0,1)$, and these points are S-stationary. Furthermore, there exists a
	W-stationary point at $\bar x=(0,0)$ which is no local minimizer.
	
	One can easily check that the associated problem \eqref{eq:relaxedMPSC_Scholtes}
	possesses a KKT point at $(\sqrt t,\sqrt t)$ for any $t\in(0,1]$. Taking
	the limit $t\downarrow 0$, this point tends to $\bar x$ which
	is, as we already mentioned above, only W-stationary for the switching-constrained problem
	of interest. Clearly, MPSC-LICQ is valid at $\bar x$.
\end{example}

Although the theoretical properties of Scholtes' relaxation approach 
do not seem to be promising in light of \eqref{eq:MPSC}, we check
the applicability of the approach. More precisely, we analyze the
restrictiveness of the assumption that a sequence of KKT points
associated with \eqref{eq:relaxedMPSC_Scholtes} can be chosen.

In order to guarantee that locally optimal solutions of the nonlinear 
relaxed surrogate problems \eqref{eq:relaxedMPSC_Scholtes}
which are located closely to the limit point from Theorem 
\ref{thm:convergence_Scholtes} are KKT points,
a constraint qualification needs to be satisfied.
Adapting \cite[Theorem~3.2]{HoheiselKanzowSchwartz2013} to the
switching-constrained situation, it is possible
to show that whenever MPSC-MFCQ is valid at a feasible point $\bar x\in X$
of \eqref{eq:MPSC}, then standard MFCQ is valid for \eqref{eq:relaxedMPSC_Scholtes}
in a neighborhood of $\bar x$.
\begin{theorem}\label{thm:MPSC-MFCQ_yields_MFCQ_Scholtes}
	Let $\bar x\in X$ be a feasible point of \eqref{eq:MPSC} where MPSC-MFCQ
	is satisfied. Then there exists a neighborhood $U\subset\R^n$
	of $\bar x$ such that MFCQ holds for \eqref{eq:relaxedMPSC_Scholtes} at all points from
	$X_{\textup S}(t)\cap U$ for all $t>0$.
\end{theorem}
\begin{proof}
	Due to the validity of MPSC-MFCQ at $\bar x$, the union
	\[
		\begin{split}
			\{\nabla g_i(\bar x)\,|\,i\in I^g(\bar x)\}\cup
				\Bigl[\{\nabla h_j(\bar x)\,|\,j\in\mathcal P\}
						&\cup\{\nabla G_l(\bar x)\,|\,l\in I^{G}(\bar x)\cup I^{GH}(\bar x)\}\\
						&\cup\{\nabla H_l(\bar x)\,|\,l\in I^{H}(\bar x)\cup I^{GH}(\bar x)\}
				\Bigr]
		\end{split}
	\]
	is positive-linearly independent. 
	Invoking Lemma \ref{lem:local_stability_of_positive_linear_independence}, there is 
	a neighborhood $U$ of $\bar x$ such that the vectors
	\[
		\begin{split}
			\{\nabla g_i(x)\,|\,i\in I^g(\bar x)\}\cup
				\Bigl[\{\nabla h_j(x)\,|\,j\in\mathcal P\}
						&\cup\{\nabla G_l(x)\,|\,l\in I^{G}(\bar x)\cup I^{GH}(\bar x)\}\\
						&\cup\{\nabla H_l(x)\,|\,l\in I^{H}(\bar x)\cup I^{GH}(\bar x)\}
				\Bigr]
		\end{split}
	\]
	are positive-linearly independent for any choice of $x\in U$. 
	
	Now, fix $t>0$ as well as $x\in X_\text{S}(t)\cap U$ and set $I^\text{a}_t(x):=I^+_t(x)\cup I^-_t(x)$. 
	Note that $t>0$ guarantees $I^+_t(x)\cap I^-_t(x)=\varnothing$. 
	Clearly, we have
	\[
		\begin{aligned}
			&\forall l\in I^H(\bar x)\colon&\quad &G_l(x)\neq 0&\quad&H_l(x)\approx 0,&\\
			&\forall l\in I^G(\bar x)\colon&&G_l(x)\approx 0&&H_l(x)\neq 0
		\end{aligned}
	\]
	if $U$ is sufficiently small. Exploiting Lemma \ref{lem:local_stability_of_positive_linear_independence}
	once more while recalling that $G$ and $H$ are continuously differentiable,
	we obtain that the vectors
	\begin{equation}\label{eq:positive_linearly_independent_vectors_MPSC_MFCQ}
		\begin{split}
			\{\nabla g_i(x)\,|\,i\in I^g(\bar x)\}\cup
				\Bigl[&\{\nabla h_j(x)\,|\,j\in\mathcal P\}\\
						&\cup\{H_l(x)\nabla G_l(x)+G_l(x)\nabla H_l(x)\,|\,l\in I^{G}(\bar x)\cap I^\text a_t(x)\}\\
						&\cup\{H_l(x)\nabla G_l(x)+G_l(x)\nabla H_l(x)\,|\,l\in I^{H}(\bar x)\cap I^\text a_t(x)\}\\
						&\cup\{\nabla G_l(x)\,|\,l\in I^{GH}(\bar x)\cap I^\text a_t(x)\}\\
						&\cup\{\nabla H_l(x)\,|\,l\in I^{GH}(\bar x)\cap I^\text a_t(x)\}
				\Bigr]
		\end{split}
	\end{equation}
	are positive-linearly independent if the neighborhood $U$ is chosen small enough.
	
	Suppose that there are vectors $\lambda\in\R^m$, $\rho\in\R^p$, and $\xi\in\R^q$ which satisfy
	\[
	 	\begin{split}
			&0=\sum\limits_{i\in I^g(x)}\lambda_i\nabla g_i(x)
				+\sum\limits_{j\in\mathcal P}\rho_j\nabla h_j(x)
				+\sum\limits_{l\in\mathcal Q}\xi_l
				\Bigl[
					H_l(x)\nabla G_l(x)+G_l(x)\nabla H_l(x)	
				\Bigr],\\
			&\forall i\in I^g(x)\colon\,\lambda_i\geq 0,\,
				\forall i\in\mathcal M\setminus I^g(x)\colon\,\lambda_i=0,\\
			&\forall l\in I^+_t(x)\colon\,\xi_l\geq 0,\\
			&\forall l\in I^-_t(x)\colon\,\xi_l\leq 0,\\
			&\forall l\in \mathcal Q\setminus I^\text a_t(x)\colon\,\xi_l=0.
		\end{split}
	\]
	In order to show the validity of MFCQ for \eqref{eq:relaxedMPSC_Scholtes} at $x$, 
	$\lambda=0$, $\rho=0$, and $\xi=0$ have to be deduced.
	We get
	\[
		\begin{split}
			&0=\sum\limits_{i\in I^g(x)}\lambda_i\nabla g_i(x)
				+\sum\limits_{j\in\mathcal P}\rho_j\nabla h_j(x)\\
			&\qquad\qquad +\sum\limits_{l\in (I^G(\bar x)\cup I^H(\bar x))\cap I^\text a_t(x)}\xi_l
				\Bigl[
					H_l(x)\nabla G_l(x)+G_l(x)\nabla H_l(x)	
				\Bigr],\\
			&\qquad\qquad +\sum\limits_{l\in I^{GH}(\bar x)\cap I^\text a_t(x)}\xi_lG_l(x)\nabla H_l(x)
							+\sum\limits_{l\in I^{GH}(\bar x)\cap I^\text a_t(x)}\xi_lH_l(x)\nabla G_l(x).
		\end{split}
	\]
	Noting that $\lambda_i\geq 0$ holds for all $i\in I^g(x)$ while $I^g(x)\subset I^g(\bar x)$ holds
	whenever $U$ is chosen sufficiently small, we obtain $\lambda=0$, $\rho=0$, 
	$\xi_l=0$ ($l\in (I^G(\bar x)\cup I^H(\bar x))\cap I^\text a_t(x)$), 
	$\xi_lG_l(x)=0$ ($l\in I^{GH}(\bar x)\cap I^\text a_t(x)$),
	and $\xi_l H_l(x)=0$ ($l\in I^{GH}(\bar x)\cap I^\text{a}_t(x)$) 
	from the positive-linear independence of the vectors in \eqref{eq:positive_linearly_independent_vectors_MPSC_MFCQ}.
	Since we have $G_l(x)\neq 0$ and $H_l(x)\neq 0$ for all $l\in I^\text a_t(x)$ from $t>0$, $\xi_l=0$
	follows for all $l\in I^\text{a}_t(x)$ since $I^G(\bar x)\cup I^H(\bar x)\cup I^{GH}(\bar x)=\mathcal Q$ 
	is valid. This yields $\xi=0$. Thus, MFCQ holds for \eqref{eq:relaxedMPSC_Scholtes} at $x$.
\end{proof}

\subsection{The relaxation scheme of Steffensen and Ulbrich}\label{sec:SteffensenUlbrich}

Here, we adapt the relaxation scheme from \cite{SteffensenUlbrich2010} for
the numerical treatment of \eqref{eq:MPSC}.
For any $t>0$, let us introduce $\phi(\cdot;t)\colon\R\to\R$ by means of
\[
	\forall z\in\R\colon
	\quad
	\phi(z;t):=
	\begin{cases}
		|z|				&	\text{if }|z|\geq t,\\
		t\,\theta(z/t)	&	\text{if }|z|<t,
	\end{cases}
\]
where $\theta\colon[-1,1]\to\R$ is a twice continuously differentiable function
with the following properties:
\begin{align*}
	&\text{(a)}\quad\theta(1)=\theta(-1)=1,&
	&\text{(b)}\quad\theta'(-1)=-1\text{ and }\theta'(1)=1,&\\
	&\text{(c)}\quad\theta''(-1)=\theta''(1)=0,&
	&\text{(d)}\quad\theta''(z)>0\text{ for all }z\in(-1,1).&
\end{align*}
A typical example for a function $\theta$ with the above properties is given by
\begin{equation}\label{eq:example_for_theta}
	\forall z\in[-1,1]\colon
	\quad
	\theta(z):=\tfrac{2}{\pi}\sin\left(\tfrac{\pi}{2}z+\tfrac{3\pi}{2}\right)+1,
\end{equation}
see \cite[Section~3]{SteffensenUlbrich2010}. Noting that the function $\phi(\cdot;t)$
is smooth for any choice of $t>0$, it can be used to regularize the 
feasible set of \eqref{eq:MPSC}. A suitable surrogate problem is given by
\begin{equation}\label{eq:relaxedMPSC_SteffensenUlbrich}\tag{P$_\text{SU}(t)$}
	\begin{aligned}
		f(x)&\,\rightarrow\,\min&&&\\
		g_i(x)&\,\leq\,0,&\qquad&i\in\mathcal M,&\\
		h_j(x)&\,=\,0,&&j\in\mathcal P,&\\
		G_l(x)+H_l(x)-\phi(G_l(x)-H_l(x);t)&\,\leq\,0,&&l\in\mathcal Q,&\\
		G_l(x)-H_l(x)-\phi(G_l(x)+H_l(x);t)&\,\leq\,0,&&l\in\mathcal Q,&\\
		-G_l(x)+H_l(x)-\phi(-G_l(x)-H_l(x);t)&\,\leq\,0,&&l\in\mathcal Q,&\\
		-G_l(x)-H_l(x)-\phi(-G_l(x)+H_l(x);t)&\,\leq\,0,&&l\in\mathcal Q.&
	\end{aligned}
\end{equation}
Its feasible set will be denoted by $X_\text{SU}(t)$ 
and is visualized in Figure \ref{fig:relaxed_feasible_set_Steffensen_Ulbrich}. 
Adapting the proof of \cite[Lemma~3.3]{SteffensenUlbrich2010}, the family
$\{X_\text{SU}(t)\}_{t>0}$ possesses the properties (P2) and (P3) from
Lemma \ref{lem:properties_of_relaxed_feasible_set}. This justifies
that \eqref{eq:relaxedMPSC_SteffensenUlbrich} is a relaxation of \eqref{eq:MPSC}.
Note that we need to introduce four inequality constraints to replace one of the original
switching constraints.

\begin{figure}[h]\centering
\includegraphics[width=4.0cm]{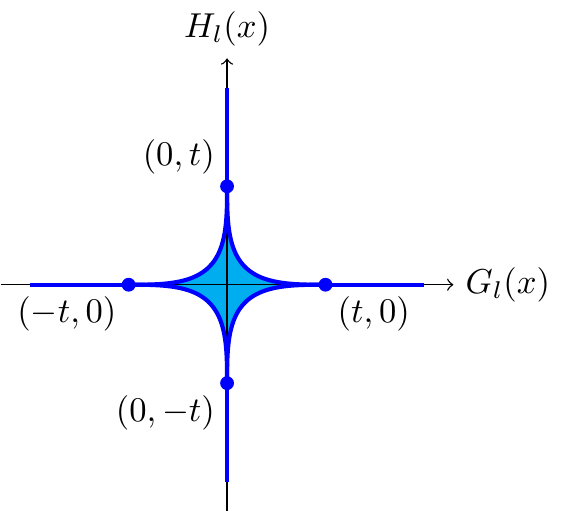}
\caption{Geometric interpretation of the relaxed feasible set $X_\text{SU}(t)$.}
\label{fig:relaxed_feasible_set_Steffensen_Ulbrich}
\end{figure}

It has been mentioned in \cite{HoheiselKanzowSchwartz2013} that the
relaxation scheme of Steffensen and Ulbrich 
computes Clarke-stationary points of MPCCs under an
MPCC-tailored version of CPLD, see \cite[Section~3.4]{HoheiselKanzowSchwartz2013} as well.
Recalling some arguments from Section \ref{sec:Scholtes}, the adapted method may
only find W-stationary points of \eqref{eq:MPSC} in general. 
The upcoming example confirms this conjecture.
\begin{example}\label{ex:SteffensenUlbrich}
	Let us consider the switching-constrained optimization problem
\begin{equation}\label{eq:ExSteffensenUlbrich}
\begin{split}
    x_1x_2-x_1-x_2&\,\rightarrow\,\min\\
    x_1^2+x_2^2-1&\,\leq\,0,\\
    x_1x_2&\,=\,0.
\end{split}
\end{equation}
	Obviously, the globally optimal solutions of this problem are given by 
	$(1,0)$ as well as $(0,1)$, and these
	points are S-stationary. Furthermore, there is a 
	W-stationary point at $\bar x=(0,0)$ which is no local
	minimizer. The global maximizers $(-1,0)$ and $(0,-1)$ do not satisfy
	any of the introduced stationarity concepts.
	
	Let us consider the associated family of nonlinear problems 
	\eqref{eq:relaxedMPSC_SteffensenUlbrich} for $t\in(0,1]$ 
	where the function $\theta$ is chosen as in \eqref{eq:example_for_theta}.
	It can easily be checked that 
	$x(t):=(\tfrac t2(1-\tfrac{2}{\pi}),\tfrac t2(1-\tfrac{2}{\pi}))$
	is a KKT point of \eqref{eq:relaxedMPSC_SteffensenUlbrich}. Note that
	$x(t)\to\bar x$ as $t\downarrow 0$, and that MPSC-LICQ holds in $\bar{x}$. However, $\bar{x}$ is only a W-stationary point of the
	switching-constrained problem \eqref{eq:ExSteffensenUlbrich}.
\end{example}

\subsection{The relaxation scheme of Kadrani, Dussault, and Benchakroun}\label{sec:Kadrani}

Finally, we want to take a closer look at the relaxation approach which
was suggested by
\cite{KadraniDussaultBenchakroun2009} for the treatment of MPCCs.
For any $t\geq 0$, let us consider the optimization problem
\begin{equation}\label{eq:relaxedMPSC_Kadrani}\tag{P$_\text{KDB}(t)$}
	\begin{aligned}
		f(x)&\,\rightarrow\,\min&&&\\
		g_i(x)&\,\leq\,0,&\qquad&i\in\mathcal M,&\\
		h_j(x)&\,=\,0,&&j\in\mathcal P,&\\
		(G_l(x)-t)(H_l(x)-t)&\,\leq\,0,&&l\in\mathcal Q,&\\
		(-G_l(x)-t)(H_l(x)-t)&\,\leq\,0,&&l\in\mathcal Q,&\\
		(G_l(x)+t)(H_l(x)+t)&\,\leq\,0,&&l\in\mathcal Q,&\\
		(G_l(x)-t)(-H_l(x)-t)&\,\leq\,0,&&l\in\mathcal Q,&
	\end{aligned}
\end{equation}
whose feasible set will be denoted by $X_\text{KDB}(t)$. The family
$\{X_\text{KDB}(t)\}_{t\geq 0}$ only satisfies property (P1) from
Lemma \ref{lem:properties_of_relaxed_feasible_set} while (P2) and (P3)
are violated in general. Thus, the
surrogate problem \eqref{eq:relaxedMPSC_Kadrani} does not induce a relaxation technique
for \eqref{eq:MPSC} in the narrower sense.
Figure \ref{fig:relaxed_feasible_set_Kadrani}
depicts that $X_\text{KBD}(t)$ is \emph{almost} disconnected, i.e.,
it is close to crumbling into four disjoint
sets for any $t>0$. This may cause serious problems when standard techniques
are used to solve the associated surrogate problem \eqref{eq:relaxedMPSC_Kadrani}.
Moreover, four inequality constraints are necessary to replace one switching constraint
from \eqref{eq:MPSC} in \eqref{eq:relaxedMPSC_Kadrani}.
\begin{figure}[H]\centering
\includegraphics[width=4.0cm]{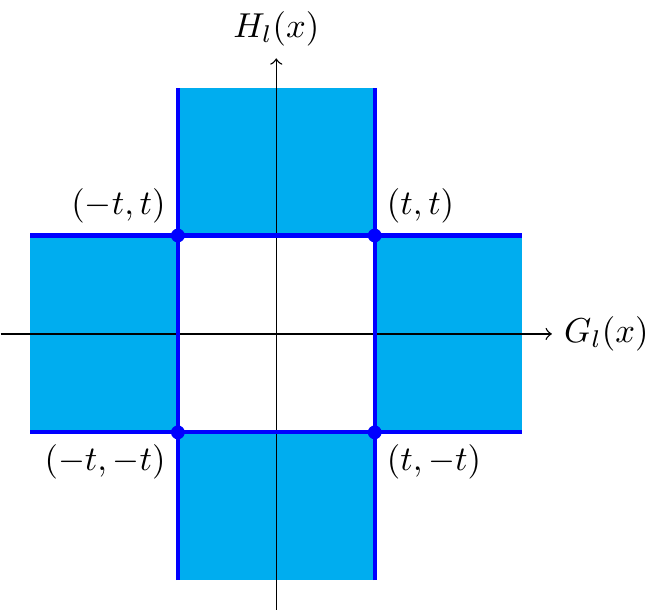}
\caption{Geometric interpretation of the relaxed feasible set $X_\text{KDB}(t)$.}
\label{fig:relaxed_feasible_set_Kadrani}
\end{figure}

On the other hand, it is clear from \cite[Section~3.3]{HoheiselKanzowSchwartz2013} 
that the regularization approach of Kadrani, Dussault, and Benchakroun
computes M-stationary points of MPCCs under an MPCC-tailored version of
CPLD at the limit point. 
Furthermore, if an MPCC-tailored LICQ holds at the limit point, then standard GCQ holds 
for the surrogate problems in a neighborhood of the point for sufficiently small relaxation parameters. 
These results are closely related to those for the relaxation approach from
\cite{KanzowSchwartz2013} which we generalized to \eqref{eq:MPSC} in Sections 
\ref{sec:relaxation} and \ref{sec:convergence_properties}.

Although we abstain from a detailed analysis of the regularization method 
which is induced by the surrogate problem \eqref{eq:relaxedMPSC_Kadrani} 
due to the aforementioned shortcomings, the above arguments motivate
the formulation of the following two conjectures.

\begin{conjecture}\label{con:convergence_Kadrani}
	Let $\{t_k\}_{k\in\N}\subset\R_+$ be a sequence of positive regularization parameters
	converging to zero. For each $k\in\N$, let $x_k\in X_{\textup{KDB}}(t_k)$ be a 
	KKT point of \hyperref[eq:relaxedMPSC_Kadrani]{\textup{(P$_{\text{KDB}}(t_k)$)}}.
	Assume that the sequence $\{x_k\}_{k\in\N}$ converges to a point $\bar x\in X$ 
	where MPSC-MFCQ holds. Then $\bar x$ is an M-stationary point of \eqref{eq:MPSC}.
\end{conjecture}

\begin{conjecture}\label{con:MPSC-LICQ_yields_GCQ_Kadrani}
	Let $\bar x\in X$ be a feasible point of \eqref{eq:MPSC} where MPSC-LICQ is satisfied.
	Then there exist $\bar t>0$ and a neighborhood $U\subset\R^n$ of $\bar x$ such that GCQ
	holds for \eqref{eq:relaxedMPSC_Kadrani} at all points from $X_\textup{KDB}(t)\cap U$ for all $t\in (0,\bar{t}]$.
\end{conjecture}

\section{Numerical results}\label{sec:numerical_results}

This section is dedicated to a detailed analysis and comparison of various numerical methods for \eqref{eq:MPSC}. 
To obtain a meaningful comparison, we apply the relaxation scheme from Section \ref{sec:relaxation_scheme_KS}
as well as a collection of other algorithms (see below) 
to multiple classes of MPSCs which possess significant practical relevance. 
The particular examples we analyze are:
\begin{itemize}
\item an either-or constrained problem with known local and global solutions
\item a switching-constrained optimal control problem involving the non-stationary heat equation in two dimensions
\item optimization problems involving semi-continuous variables, in particular special instances of portfolio optimization
\end{itemize}
For each of the examples, we first provide an overview of the corresponding problem structure, 
and then give some numerical results. 
To facilitate a quantitative comparison of the used algorithms, we use performance profiles (see \cite{DolanMore2002}) based on the computed function values.

\subsection{Implementation}

The numerical experiments in this section were all done in MATLAB R2018a. 
The particular algorithms we use for our computations are the following:
\begin{itemize}[leftmargin=6em]
\item[\textbf{KS:}] the adapted Kanzow--Schwartz relaxation scheme from this paper
\item[\textbf{FMC:}] the \texttt{fmincon} function from the MATLAB Optimization Toolbox
\item[\textbf{SNOPT:}] the SNOPT nonlinear programming solver from \cite{GillMurraySaunders2002}, called through the TOMLAB programming environment
\item[\textbf{IPOPT:}] the IPOPT interior-point algorithm from \cite{WaechterBiegler2006}
\end{itemize}
The overall implementation is done in MATLAB, and each algorithm is called with user-supplied gradients of the objective functions and constraints. The stopping tolerance for all algorithms is set to $10^{-4}$ (although it should be noted that the methods use different stopping criteria, i.e., they impose the accuracy in different ways). For the KS algorithm, the relaxation parameters are chosen as $t_k:=0.01^k$, and the method is also terminated as soon as $t_k$ drops below $10^{-8}$. Finally, to solve the relaxed subproblems in the relaxation method, we employ the SNOPT algorithm with an accuracy of $10^{-6}$.

Judging by past experience with MPCCs, the SNOPT algorithm can be expected to rival the relaxation scheme in terms of robustness.
To accurately measure the performance of the solvers, it is important to note that MPSCs can, in general, admit a substantial amount of local minimizers.
Therefore, the robustness is best measured by comparing the obtained function values (using different methods and starting points) with the globally optimal function value---if the latter is known; otherwise, a suitable approximate is used.
To avoid placing too much emphasis on the accuracy of the final output (which does not make sense since the algorithms use completely different stopping criteria), we use the quantity
\begin{equation}\label{Eq:PerfProfMetric}
    Q_\delta(x):=
    \begin{cases}
        f(x)-f_{\min}+\delta, & \text{if $x$ is feasible within tolerance}, \\
        +\infty, & \text{otherwise}
    \end{cases}
\end{equation}
as the base metric for the performance profiles, where $x$ is the final iterate of the given algorithmic run, $f_{\min}$ the (approximate) global minimal value of the underlying problem, and $\delta\ge 0$ is an additional parameter which reduces the sensitivity of the values to numerical accuracy.
We have found that an appropriate choice of $\delta$ can significantly improve the meaningfulness of the results.

\subsection{Numerical examples}

The following pages contain three examples of MPSCs. 
In Section \ref{sec:exp2}, we deal with an \emph{either-or constrained} problem, which can be reformulated as an MPSC, see \cite{Mehlitz2018}.
Section~\ref{sec:exp3} is dedicated to a switching-constrained optimal control problem based on the framework from \cite{ClasonRundKunisch2017}.
Finally, in Section~\ref{Sec:Portfolio}, we deal with a class of optimization problems with \emph{semi-continuous variables}, 
which can again be reformulated as MPSCs. 
This section contains a particular example from portfolio optimization which originates from \cite{FrangioniGentile2007}.

%
%

\subsubsection{An either-or constrained example}\label{sec:exp2}

Let us consider the optimization problem
\begin{equation}\label{eq:exp2}\tag{E$2$}
	\begin{split}
		(x_1-8)^2+(x_2+3)^2&\,\rightarrow\,\min\\
		x_1-2x_2+4\,\leq\,0\;\lor\;x_1-2&\,\leq\,0,\\
		x_1^2-4x_2\,\leq\,0\;\lor\;(x_1-3)^2+(x_2-1)^2-10&\,\leq\,0.
	\end{split}
\end{equation}
Here, $\lor$ denotes the logical ``or''. The feasible set of this program is visualized in Figure~\ref{fig:exp2}.
{It is easily seen that \eqref{eq:exp2} possesses the unique global minimizer $\bar x=(2,-2)$ and another local minimizer 
$\tilde x=(4,4)$.
\begin{figure}[h]\centering
\includegraphics[width=0.45\textwidth]{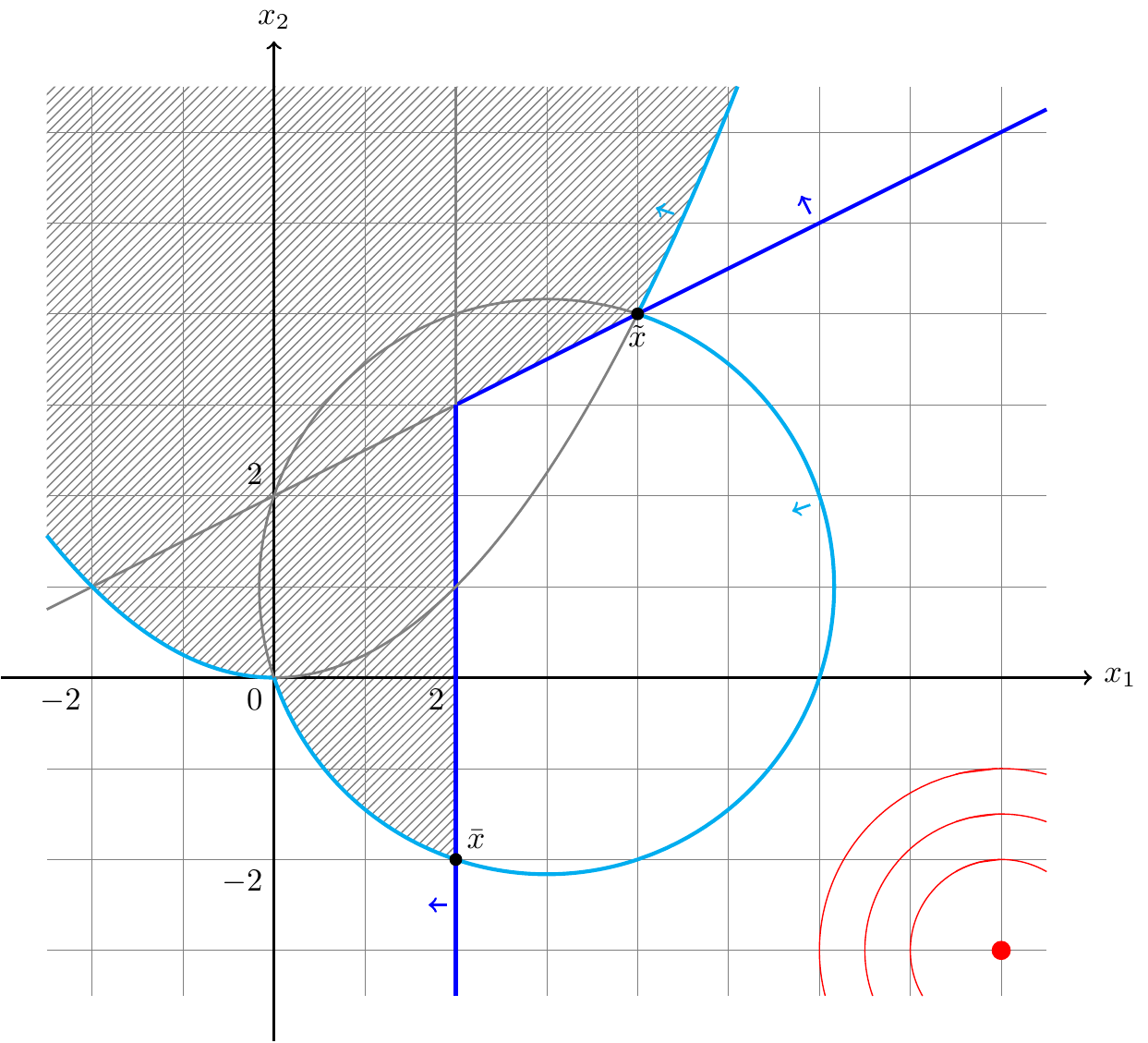}
\hspace{1em}
\includegraphics[width=0.45\textwidth]{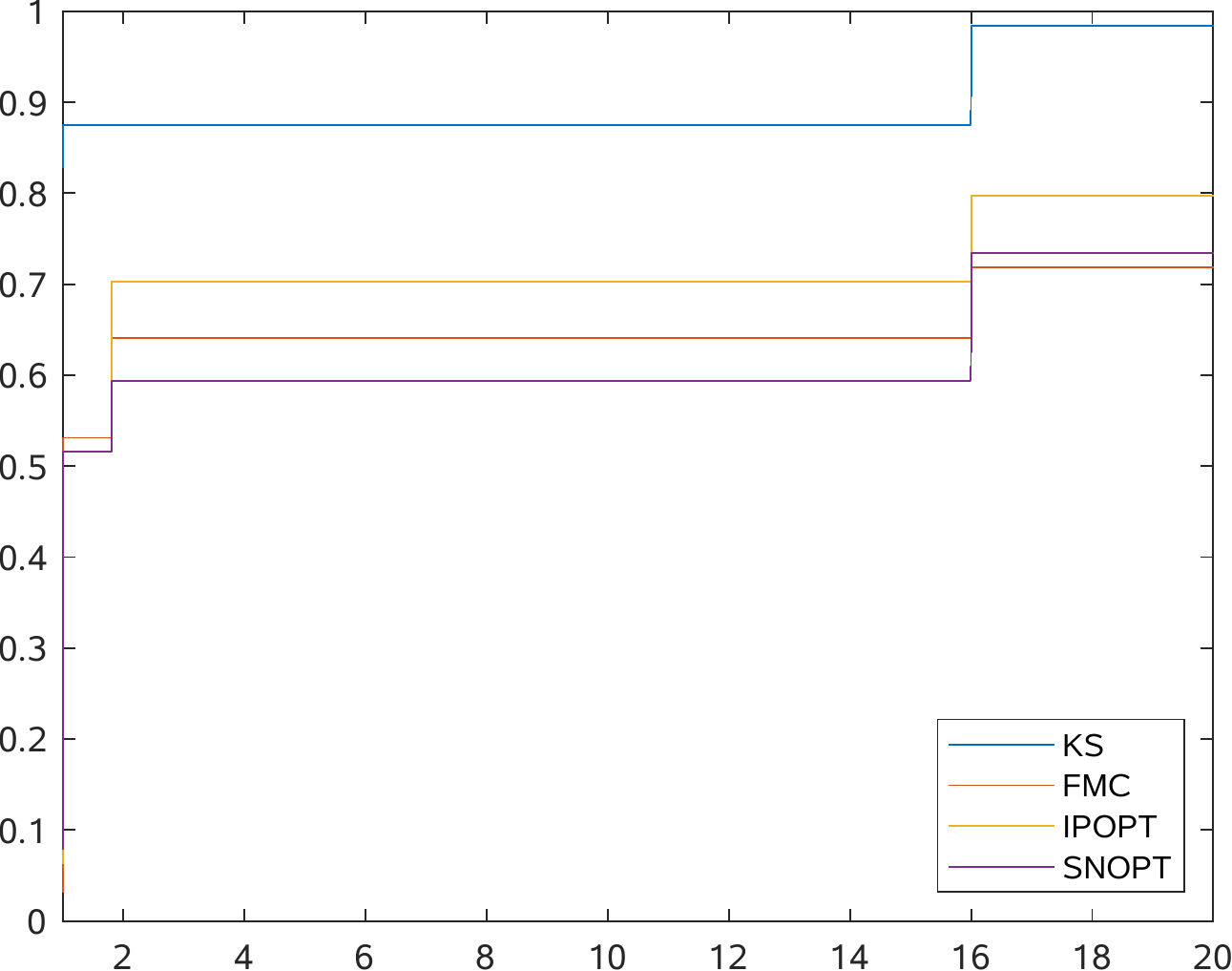}
\caption{The either-or constrained problem from Section~\ref{sec:exp2}: feasible set, some level sets, and local minimizers (left), performance profile (right).}
\label{fig:exp2}
\end{figure}
Arguing as in \cite[Section~7]{Mehlitz2018}, we can transform \eqref{eq:exp2} into a switching-constrained
optimization problem by introducing additional variables:
\begin{equation}\label{eq:exp2_switching}
	\begin{split}
		(x_1-8)^2+(x_2+3)^2&\,\rightarrow\,\min\limits_{x,z}\\
		z_1,z_2,z_3,z_4&\,\leq\,0,\\
		(x_1-2x_2+4-z_1)(x_1-2-z_2)&\,=\,0,\\
		(x_1^2-4x_2-z_3)((x_1-3)^2+(x_2-1)^2-10-z_4)&\,=\,0.
	\end{split}
\end{equation}
Note that the local minimizers of \eqref{eq:exp2} can be found among the local minimizers of \eqref{eq:exp2_switching}
choosing suitable values for the variable $z$, see \cite[Section~7.1]{Mehlitz2018}.

The algorithms in question are each tested with the starting points in the set $\{0,1\}^6$, 
which makes for a total of $64$ starting points. 
The resulting performance profile can also be found in Figure~\ref{fig:exp2}; 
it is based on the metric \eqref{Eq:PerfProfMetric} with $\delta:=1$. 
Clearly, the KS relaxation method is the most robust of the four algorithms, 
finding the best function values (among the tested algorithms) in more than 80\% of the test runs.

\subsubsection{Switching-constrained optimal control}\label{sec:exp3}

Here, we intend to solve a switching-constrained optimal control problem with the proposed relaxation method.
The underlying example is taken from \cite[Section~5.2]{ClasonRundKunisch2017}. 

Let $I:=(0,T)$, with $T:=10$ the final time, $\Omega:=(-1,1)^2$, and let $\Gamma$ be the boundary of $\Omega$. 
Furthermore, we define $\Omega_u:=(-1,0]\times(-1,1)$ as well as $\Omega_v:=(0,1)\times(-1,1)$. 
Let us consider the optimal control of the non-stationary heat equation with zero initial and Neumann boundary conditions given below:
\begin{equation}\label{eq:heat_equation}
\begin{aligned}
    \partial_t y(t,\omega)-\Delta_\omega y(t,\omega)-\tfrac{1}{10}\chi_{\Omega_u}(\omega)u(t)-\tfrac{1}{10}\chi_{\Omega_v}(\omega)v(t)&\,=\,0&\quad&\text{a.e.\ on }I\times\Omega,\\
    \vec{\mathbf n}(\omega)\cdot\nabla_\omega y(t,\omega)&\,=\,0&&\text{a.e.\ on }I\times\Gamma,\\
    y(0,\omega)&\,=\,0&&\text{a.e.\ on }\Omega.
\end{aligned}
\end{equation}
Here, $\chi_A\colon\Omega\to\R$ denotes the characteristic function of a measurable set $A\subset\Omega$ which equals $1$
on $A$ and vanishes otherwise.
Let $y_\text{d}\in L^2(I;H^1(\Omega))$ be the solution of the state equation associated with the desirable
control functions $u_\text{d},v_\text{d}\in L^2(I)$ given by
\begin{equation*}
    \forall t\in I\colon\quad u_\text{d}(t):=20\sin^4 (2\pi t/T ),\quad v_\text{d}(t):=10\cos^4(1.4\pi t/T).
\end{equation*}
All feasible controls $u,v\in H^1(I)$ shall satisfy the switching requirement
\begin{equation}\label{eq:switching_controls}
    u(t)v(t)=0\quad\text{a.e.\ on }I.
\end{equation}
Note that $u_\text{d}$ and $v_\text{d}$ violate this switching condition.
We aim to find the minimum of the objective function defined by
\begin{equation}\label{eq:exp3}
\begin{aligned}
    \mathcal{J}(y,u,v):=\tfrac{1}{2}\norm{y-y_\text d}{L^2(I;L^2(\Omega))}^2
		& +\tfrac{\alpha}{2}\left(\norm{u}{L^2(I)}^2+\norm{v}{L^2(I)}^2\right) \\
		& +\tfrac{\beta}{2}\left(\norm{\partial_t u}{L^2(I)}^2+\norm{\partial_t v}{L^2(I)}^2\right)
\end{aligned}
\end{equation}
with respect to $(y,u,v)\in L^2(I;H^1(\Omega))\times H^1(I)\times H^1(I)$ such that
$(u,v)$ satisfy the switching requirement \eqref{eq:switching_controls} while $y$ solves
the associated state equation \eqref{eq:heat_equation}.
We chose $\alpha:=10^{-6}$ and $\beta:=10^{-5}$ for our experiments.

For the numerical solution of the problem, the domain $\Omega$ is tessellated using the function \texttt{generateMesh} from the MATLAB PDE toolbox 
and a discretization tolerance of $h:=10^{-1}$. 
The time interval $I$ is subdivided into equidistant intervals of size $\tau:=10^{-1}$. 
Both the spatial and temporal discretizations use standard piecewise linear (continuous) finite elements, 
which leads to a conforming approximation of the $H^1$-norm in \eqref{eq:exp3}.

After discretization, the problem turns into a finite-dimensional MPSC comprising the variables $\vec u,\vec v\in\R^{101}$, a quadratic objective function, 
and the switching constraints $\vec u_i\vec v_i=0$ for all $i=1,\ldots,101$. 
These correspond to the simple constraint mappings $G(\vec u,\vec v):=\vec u$ and $H(\vec u,\vec v):=\vec v$. 
Note that the feasible set can be seen as the union of $2^{101}$ convex ``branches'' 
(obtained by setting either $\vec u_i=0$ or $\vec v_i=0$ for each $i=1,\ldots,101$).
Hence, the problem can be expected to admit a substantial amount of local minimizers, 
and it is unrealistic to expect algorithmic implementations to find the global minimizer, even when tested with a large number of initial points. 
To facilitate a quantitative comparison of our numerical algorithms (as in the previous section), 
we use the following heuristic to obtain an upper estimate of the optimal value: 
using a coarser time discretization (with $\tau=0.4$), we compute the \emph{exact} global minimizer of the resulting problem 
by minimizing the objective over each of the branches induced by the switching constraints. 
The corresponding global minimizer is then lifted to the finer time grid (with $\tau=0.1$) by linear interpolation, 
and the resulting point is used as an initial guess for all the used algorithms. 
The resulting estimate of the optimal value is $0.2997$, and the associated controls are depicted in Figure~\ref{fig:possible_global_minimizer}.

For the numerical tests, we generated $1000$ starting points with coordinates chosen randomly in the interval $[0,10]$. 
The performance profile was constructed by using the metric \eqref{Eq:PerfProfMetric} with $\delta:=0$, 
and it too can be found in Figure~\ref{fig:possible_global_minimizer}.

As in the previous example, the relaxation method turns out to be the most robust of the tested algorithms, 
finding lower function values in around 70\% of the test runs. 
When analyzing the results in more detail, it turns out that, as expected, 
the algorithms found an exorbitant amount of distinct points (possibly local minimizers). 
Interestingly, however, the associated function values actually lie quite close to each other. 
This explains the $x$-axis scaling in the performance profile plot.

\begin{figure}[h]\centering
\includegraphics[width=0.45\textwidth]{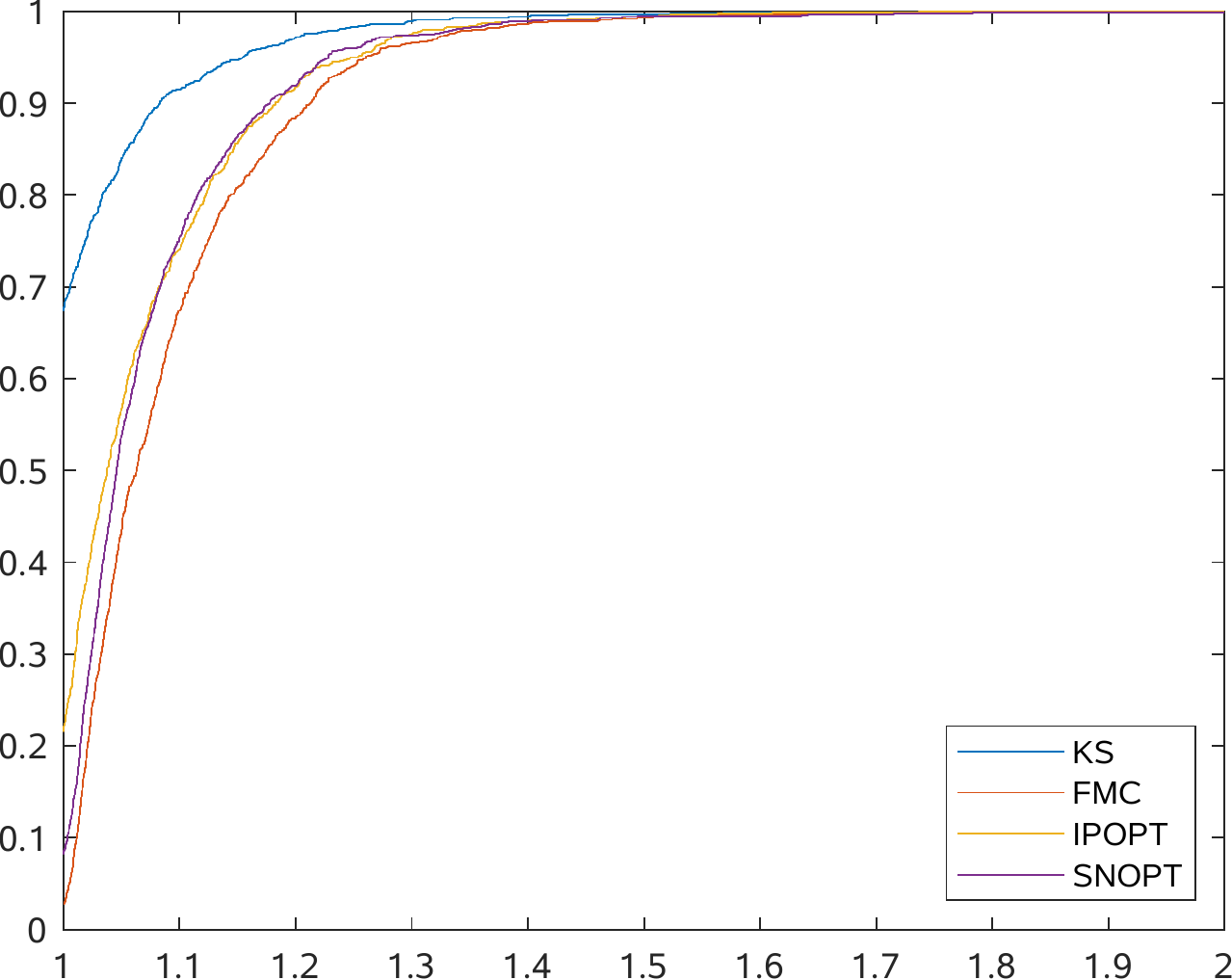}
\hspace{1em}
\includegraphics[width=0.45\textwidth]{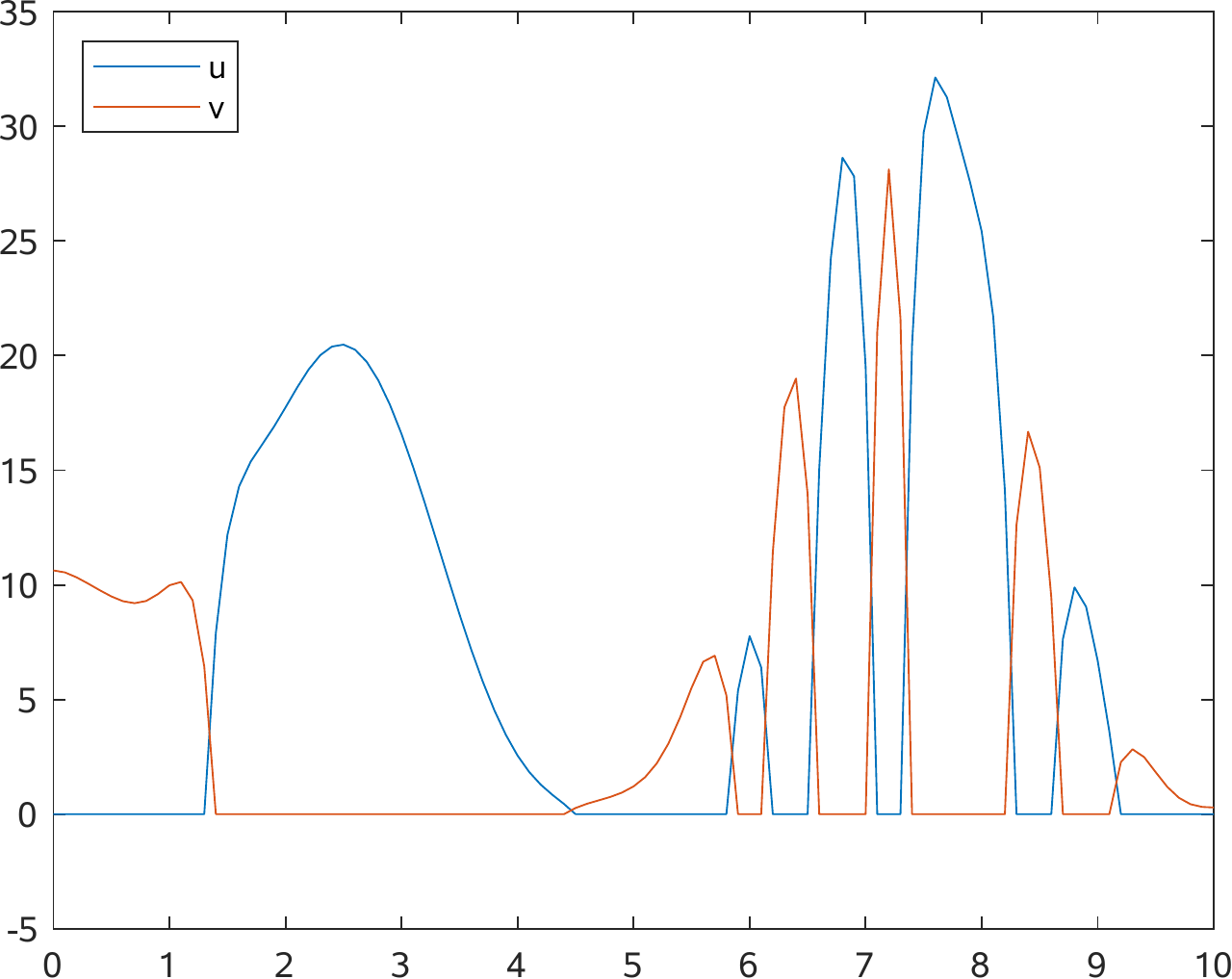}
\caption{Performance profile and (possible) global minimizer for the switching-constrained optimal control problem from Section~\ref{sec:exp3}.}
\label{fig:possible_global_minimizer}
\end{figure}

\subsubsection{Semi-continuous variables}\label{Sec:Portfolio}

In many optimization scenarios, it is desirable that a nonnegative decision variable is either exactly zero 
or contained in some interval whose lower bound is strictly positive. 
This is the case, for instance, in production planning, portfolio optimization, 
compressed sensing in signal processing, and subset selection in regression. 
More details can be found in \cite{BurdakovKanzowSchwartz2016,Sun2013}, and the references therein.

Given a decision variable $x\in\R^n$ and bounds $\ell,u\in\R^n$, $\ell\le u$, a requirement of the above form can be reformulated as the either-or type constraints
\begin{equation}\label{Eq:SemicontinuousConstraint}
    x_i=0 \; \lor \; x_i\in [\ell_i,u_i],\qquad i=1,\ldots,n.
\end{equation}
In this context, the variable $x$ is often called \emph{semi-continuous} since it is required to lie in some continuous interval, 
except for the outlier case when it is equal to zero. 
(One might also be inclined to call $x$ a \emph{semi-discrete} variable, but we have not seen this terminology elsewhere in the literature.)

Constraint systems of the form \eqref{Eq:SemicontinuousConstraint} can be reformulated as switching constraints by using slack variables. 
Indeed, there are two ways of doing so: 
On the one hand, we could introduce two nonnegative slack variables to transform the box constraints 
in \eqref{Eq:SemicontinuousConstraint} into equality constraints; 
this procedure eventually yields an MPSC with $2 n$ switching constraints. 
On the other hand, assuming that $u_i\ge 0$ holds for all $i=1,\ldots,n$ 
(which is the case in nearly all relevant applications), 
we can simply treat the requirement $x_i\le u_i$ as a standard inequality constraint which should be fulfilled at all times. 
Clearly, if $x_i=0$ is valid, then the inequality $x_i\le u_i$ holds automatically, 
so that we can rewrite \eqref{Eq:SemicontinuousConstraint} as
\begin{equation*}
	\begin{aligned}
    x_i&\, \le\,u_i, 				\quad	&&i=1,\ldots,n,&\\
    x_i=0\,\lor\,x_i&\,\ge \ell_i,	\quad 	&&i=1,\ldots,n.&
    \end{aligned}
\end{equation*}
Using a single slack vector $y\in\R^n$, we can now rewrite this system as
\begin{equation*}
	\begin{aligned}
    	x&\,\le\,u,&&&\\
    	y&\,\ge\,0,&&&\\
    	x_i(x_i-\ell_i-y_i)&\,=\,0,\quad&&i=1,\ldots,n.&
    \end{aligned}
\end{equation*}
In the notation of our general framework \eqref{eq:MPSC}, this corresponds to the switching mappings
\begin{equation*}
    G(x,y):=x \quad\text{and}\quad H(x,y):=x-\ell-y.
\end{equation*}
The inequality constraints $x\le u$ and $y\ge 0$ can be implemented as components of the mapping $g$ 
(which may contain other constraints depending on the particular problem). 
Note that the above reformulation approach only results in $n$ switching constraints.

\begin{figure}[h]\centering
\includegraphics[scale=0.5]{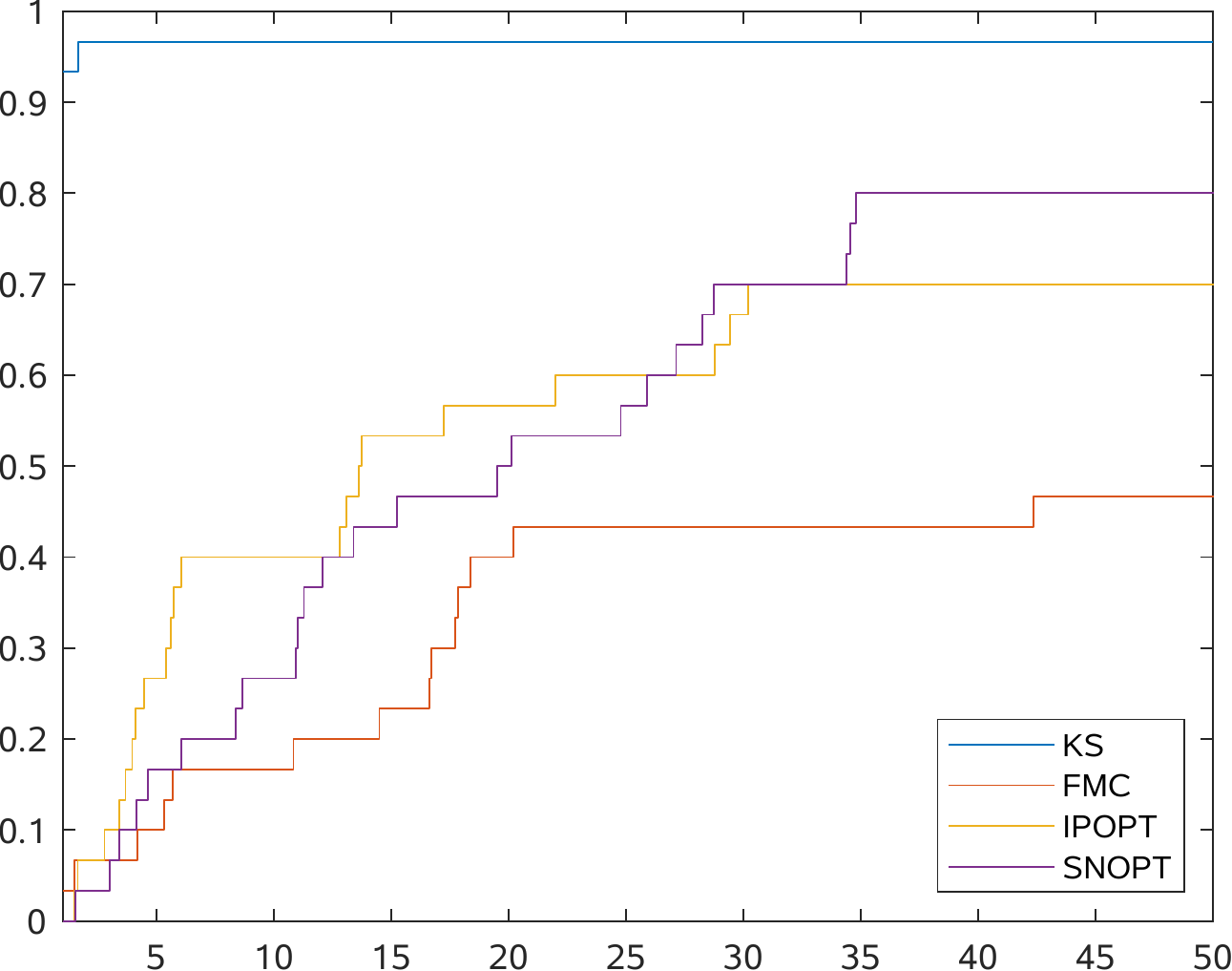}
\caption{Performance profile for the portfolio optimization problem from Section~\ref{Sec:Portfolio}.}
\label{fig:portfolio}
\end{figure}

We now present a concrete example of portfolio optimization based on the test examples in \cite{FrangioniGentile2007}. The problems in this reference have the form
\begin{equation}\label{Eq:PortfolioExample}
\begin{aligned}
    x^{\top}Q x & \,\to\,\min \\
    e^{\top}x &\, =\,1, \\
    \mu^{\top}x &\, \ge \,\rho, \\
    \quad x_i\,=\,0\;\lor\,x_i&\,\in\,[\ell_i,u_i],\quad i=1,\ldots,n,
\end{aligned}
\end{equation}
with randomly generated $Q\in\R^{n\times n}$, $\mu,\ell,u\in\R^n$, and $\rho\in\R$. 
Here, $e\in\R^n$ represents the all-ones vector.
More details can be found in \cite{FrangioniGentile2007} and on their webpage \url{http://www.di.unipi.it/optimize/Data/MV.html}. 
The particular examples we chose are the $30$ instances with size $200$. 
The corresponding problems \eqref{Eq:PortfolioExample} are reformulated as MPSCs by means of the aforementioned procedure, 
and the resulting problems are then attacked by the four test algorithms in question. 
Figure~\ref{fig:portfolio} depicts the resulting performance profile based on the metric \eqref{Eq:PerfProfMetric} with $\delta:=0$.

For this particular problem class, it turns out that the performance advantage of the relaxation method 
is particularly large when compared to its non-relaxed counterparts. 
In 28 out of 30 examples, the KS algorithm finds the best function value among the tested methods; 
in addition, it also seems to find feasible points much more reliably than the other algorithms.

\section{Final remarks}\label{sec:final_results}

In this paper, we have presented a relaxation method for the solution of mathematical programs with switching constraints (MPSCs). 
Our theoretical analysis yields strong convergence properties for the method; 
in particular, the algorithm computes M-stationary points of MPSCs under a problem-tailored constraint qualification (MPSC-NNAMCQ) 
which is weaker than MPSC-MFCQ. 
The numerical experiments include a wide array of practically relevant problems 
and demonstrate the computational efficiency of the proposed algorithm.

In addition, we have conducted a dedicated analysis for other relaxation schemes which are known in the MPCC literature 
and can be carried over to the setting of switching constraints. 
In particular, adapted versions of the relaxation schemes of Scholtes as well as Steffensen and Ulbrich
are shown to converge to weakly stationary points only, even if fairly strong regularity properties such as MPSC-LICQ are satisfied.


\bibliographystyle{plainnat}
\bibliography{references}

\end{document}